\documentclass[11pt,a4paper,leqno]{amsart}%{amsproc}

\usepackage{amsmath,amsfonts,amssymb,amsthm}
\usepackage[T1]{fontenc}

\advance\textwidth by 3.2cm %%3.6
\advance\oddsidemargin by -1.6cm \advance\evensidemargin by -1.6cm

\input colordvi

\newcommand\new[1]{}
\newcommand\nn{\nonumber}

\usepackage{amssymb}

\usepackage{times}

\usepackage{mathrsfs}

\usepackage{verbatim}

\input{colordvi}

\theoremstyle{plain}

\newtheorem{theorem}{Theorem}[section]

\theoremstyle{remark}

\newtheorem{remark}[theorem]{Remark}

\newtheorem{example}[theorem]{Example}

\theoremstyle{plain}

\newtheorem{corollary}[theorem]{Corollary}

\newtheorem{lemma}[theorem]{Lemma}

\newtheorem{proposition}[theorem]{Proposition}

\newtheorem{definition}[theorem]{Definition}

\newenvironment{del}[1]{}{\par\vspace{6pt}}

\numberwithin{equation}{section}

\newcommand{\lb}{\langle}

\newcommand{\rb}{\rangle}

\newcommand{\embed}{\hookrightarrow}

\newcommand\lin{\operatorname{lin}}

\newcommand\diver{\operatorname{div}}

\newcommand\id{\operatorname{I}}

\newcommand{\eps}{\varepsilon}

\newcommand{\al}{\alpha}

\newcommand{\la}{\lambda}

\def\Rnu{\mathbb{R}}

\def\Rn{\mathbb{R}^n}

\def\liml{\lim\limits}

\def\liminfl{\liminf\limits}

\def\suml{\sum\limits}

\def\supl{\sup\limits}

\def\intl{\int\limits}

\def\supl{\mathop{\sup}\limits}

\begin{document}
\title[Backward Uniqueness for  parabolic SPDEs]{Backward Uniqueness and the existence of the spectral limit for some parabolic SPDEs}

\author[Z. Brze\'{z}niak and  M. Neklyudov]{Zdzis{\l}aw Brze\'{z}niak and  Misha Neklyudov}
\address{Department of Mathematics\\
The University of York\\
Heslington, York YO10 5DD, UK} \email{zb500@york.ac.uk}
\address{Department of Mathematics\\
The University of York\\
Heslington, York YO10 5DD, UK}
\email{mn505@york.ac.uk,misha.neklyudov@gmail.com}

\date{\today}

\begin{abstract}
The aim of this article is to study the asymptotic behaviour for
large times of solutions to a certain class of stochastic partial
differential equations of parabolic type. In particular, we will
prove the backward uniqueness result and the existence of the
spectral limit for abstract SPDEs and then show how these results
can be applied to  some concrete linear and nonlinear SPDEs. For
example,  we will consider linear parabolic SPDEs with gradient
noise and stochastic NSEs with multiplicative noise. Our results
generalize the results proved in \cite{[Ghidaglia-1986]} for
deterministic PDEs.

\end{abstract}

\maketitle \tableofcontents

\thanks{}

\section{Introduction and formulation of  the main results}
\label{sec:into}

The question of uniqueness of solutions to both deterministic and
stochastic, both ordinary and partial, differential equations, is
quite fairly well understood. There are plenty of positive results
and there are some counterexample. This area is too vast to make
any attempt of listed relevant papers. The question of backward
uniqueness is equivalent to classical (i.e. forward)  uniqueness
in the case of ordinary differential equations. In the case of
stochastic differential equations the backward uniqueness is
closely related to the question of existence of a stochastic flow.
In fact, the latter implies the former and the latter has been
extensively studied since the pioneering works by Blagovescenski
and  Freidlin \cite{Bl_Freidlin_1961}. However, parabolic
equations can only be solved forward and the backward uniqueness
is completely unrelated to the forward uniqueness. To our
knowledge, first results on backward uniqueness are due to Lees
and Protter \cite{Lees+Protter_1961} and Mizohata
\cite{Mizohata_1958}. This has been followed by a long series of
papers, often using very different approaches, see e.g. Ghidaglia
\cite{[Ghidaglia-1986]} and Escauriaza et al \cite{Esc+S+Sv_2003}.
Primary applications of backward uniqueness  is the study of long
time behaviour of the solutions but there are also natural
applications to control theory, see e.g. \cite{Micu_Zuazua_2001}.
As in the case of PDEs also in the case of stochastic PDEs,   the
existence of a flow does not implies backward uniqueness.
Furthermore, there are only few known examples of SPDEs which have
flows, see e.g. \cite{Flandoli_Sch_1990},  \cite{Brz_Fl_1991},
\cite{Moh+Zh+Zh_2008} and references therein. Hence the question
of backward uniqueness for  SPDEs  of parabolic type is even more
important that the similar question for deterministic parabolic
PDEs. Possible applications are paramount, let us just mention the
most obvious: long time behaviour of solutions and control theory.
The backward uniqueness we prove should be applicable to study
regularity of the local spectral manifolds constructed in
Flandoli-Schauml{\"o}ffell, c.f. Ruelle \cite{Ruelle_1982},
Foias-Saut \cite{Foias+Saut_1984_s} and \cite{Brz_1991}.

To the best of our knowledge our paper is the first one in which
such an important question is being investigated. As mentioned in
the abstract, in our paper we   generalize the results proved in
\cite{[Ghidaglia-1986]} for deterministic PDEs. One of the
difficulties with   extending   the results from
\cite{[Ghidaglia-1986]} to the stochastic case is  that the
standard It\^o  formula is not directly applicable to the case
considered in this article. We use certain approximations to
overcome this problem.
%Another difficulty is that conditions
%1.3-1.4, p.779 in \cite{[Ghidaglia-1986]} have no natural
%counterpart in the stochastic case.  We have only conditions
%\eqref{eqn:NonlinearityAss}, \eqref{eqn:NonlinearityAss-n}. As a
%result, we require  rather strong assumptions on the regularity of
%solutions in the case of stochastic equations with quadratic
%nonlinearity. In the same time, this problem does not appear in
%the case of linear stochastic equations or if nonlinearity has no
%more than "linear growth".

Let us now briefly present construction of the paper. At the end
of this Introduction we will present notation, assumptions  and
the main results. In section $2$ we state the proof of the Theorem
about backward uniqueness of SPDEs. The argument is based on
stochastic version of logarithmic convexity approach. The proof is
separated in the series of Lemmas for convenience of the reader.
Proof of the Lemma \ref{lem:AuxRes_2} can be omitted in the first
reading. Section $3$ contains proof of the Theorem about existence
of spectral limit. The main difference of the proof comparing to
backward uniqueness Theorem is that we use comparison Theorem for
one dimensional diffusions to derive main a priori estimate. In
section $4$ we present examples of application of our result. It
includes quite wide class of linear SPDEs and one example of
application to 2D NSE with multiplicative noise of special form.
In appendix we collect some auxiliary results applied in the
proofs.

Assume that $(\Omega,\mathcal{F},\{\mathcal{F}_t\}_{t\geq
0},\mathbb{P})$ is a  complete filtered probability space and
$(w_t)_{t\geq 0}$ is an $\mathbb{R}^n$-valued $\mathbb{F}$-Wiener
process, where $\mathbb{F}=\{\mathcal{F}_t\}_{t\geq 0}$.  We
assume that
\begin{equation}\label{eqn-Gelfand}
V\subset H\cong H^\prime\subset V^\prime
\end{equation}
is a Gelfand triple of Hilbert spaces. The norm in $V$,
respectively $H$, will be denoted by $\Vert\cdot \Vert$,
respectively   by $\vert\cdot \vert$. The scalar product in $H$
(resp. $V$) will be denoted by $(\cdot,\cdot)_H$ (resp.
$(\cdot,\cdot)_V$) and the duality pairing between $V'$ and $V$
will be denoted by $<\cdot,\cdot>_{V',V}$. We will omit the
indexes where no uncertainty appears. The Banach space of trace
class operators in $H$ will be denoted $\mathcal{T}_1(H)$.

We assume that $A(t)$, $t\in [0,\infty)$ is a family of bounded
linear operators from $V$ to $V^\prime$ such that the sets
$D(A(t))=:D(A)$, $t\in [0,\infty)$ are independent of time and
$(B_k(t))_{k=1}^{n}$, $t\in [0,\infty)$ is a family of  bounded
linear operators both from $V$ to $H$ and from $H$ to $V^\prime$.
Let us set
\begin{equation}\label{eqn-A_tilde}
\tilde{A}(t)=A(t)-\frac{1}{2}\suml_kB_k(t)^*B_k(t), t\in[0,T].
\end{equation}
 We will
assume that the sets  $D(\tilde{A}(t))=:D(\hat{A})$, $t\in
[0,\infty)$ are independent of time. Here $\hat{A}:D(\hat{A})\to
H$ is selfadjoint strictly positive operator defined by
$$(\hat{A}u,v)_H:=(u,v)_V,u,v\in V.$$ Then, see
\cite{[LionsMagenes-72]}, p. 9-10, $V=D(\hat{A}^{1/2})$ and there
exists an orthonormal basis $\{e_i\}_{i\geq 1}\subset D(\hat{A})$
of eigenvectors of $\hat{A}$ in $H$.
%Notice that $e_i\in D(\hat{A})$ for all $i$. %in which operator $\tilde{A}$ is diagonal
%(such basis always exists because $\tilde{A}$ is selfadjoint and
%positive)
For $N\in\mathbb{N}$ let  $P_N:H\to H$ be the orthogonal
projection onto the space $H_N=\lin\{e_1,\cdots,e_N\}$ and let
$Q_N=\id-P_N$. We can notice that $P_N:V\mapsto V$ and
$||P_N||_{\mathcal{L}(V,V)}\leq 1$. Denote $P_N':V'\to V'$ adjoint
operator to $P_N$ w.r.t. duality between $V$ and $V'$. Define a
linear operator $\tilde{A}_N=P_N'\tilde{A}P_N$.

We assume that a map  $$F:[0,T]\times V \to V'$$ is such that it
maps $[0,T]\times D(A)$ to $H$.

If $X$ is a separable Hilbert space then by
$\mathcal{M}^p(0,T;X)$, $p\geq 1$,  we will understand   the space
of all progressively measurable stochastic processes  $\xi
:[0,T]\times\Omega\to X$, or rather their equivalence classes,
such that
$$
\mathbb{E}\intl_0^T|\xi(s)|_X^p\,ds<\infty.
$$
\begin{definition}\label{def:Solution}
A progressively measurable $H$-valued stochastic process $u(t)$,
$t\geq 0$ is a solution of the problem
\begin{equation}\label{eqn-1.1}
\left\{
\begin{aligned}
du(t)&+(A(t)u(t)+F(t,u(t)))\,dt+\suml_{k=1}^nB_k(t)u(t)\,dw^k(t)=0,\,t\geq
0\cr u(0)&=u_0
\end{aligned}
\right.
\end{equation}
if and only if for each $T>0$ $u\in \mathcal{M}^2(0,T;V)\cap
L^2(\Omega,C([0,T],H))$  and for any $t\in[0,\infty)$, the
equality
\begin{equation}\label{eqn-1.2}
u(t)=u_0-\intl_0^t(A(s)u(s)\,ds+F(s,u(s)))\,ds-\suml_{k=1}^n\intl_0^tB^k(s)u(s)\,\,
dw(s)
\end{equation}
is satisfied $\mathbb{P}$-a.s..
%\Red{I see a problem here as the It\^{o}
%integral is defined only $\mathbb{P}$-a.s. so how can the above equality holds
%for all $\omega\in\Omega^\prime$? $\Omega^\prime$ is a subset of measure $1$
%of  the space $\Omega$.}\Blue{This needs some attention in the
%formulation. Maybe one needs to mention a version of the It\^{o}
%integral?}
\end{definition}
The following assumptions   will be used throughout the paper.

\begin{trivlist}
\item[\textbf{(AC0)}] The maps $A:[0,T]\to
\mathcal{L}(V,V^\prime)$ and $B^k:[0,T]\to \mathcal{L}(V,H)$,
$k=1,\cdots,n$ are strongly measurable and bounded, i.e. for each
$v\in V$, the functions $[0,T]\ni t\mapsto A(t)v\in V^\prime$ and
$[0,T]\ni t\mapsto B^k(t)v\in H$, $k=1,\cdots,n$ are measurable
and bounded. \item[\textbf{(AC1)}] There exists a map
$\tilde{A}^\prime:[0,T]\to\mathcal{L}(V,V^\prime)$ such that for
all $\phi\in V$ and $\psi\in V^\prime$,   $\lb
\tilde{A}^\prime(\cdot)\phi,\psi\rb\in L^1(0,T)$ and, in the weak
sense,
\begin{eqnarray}
\frac{d}{dt}\lb \tilde{A}(\cdot)\phi,\psi\rb &=&\lb
\tilde{A}^\prime(\cdot)\phi,\psi\rb.\label{eqn:ATildeTimeDerivAss}
\end{eqnarray}
\item[\textbf{(AC2)}] There exists constants $\al>0$ and
$\lambda\in\Rnu$ such that for all $t\in [0,T]$,
\begin{eqnarray}
2\lb A(\cdot)u,u\rb+\lambda |u|^2\geq\al \Vert u\Vert
^2+\suml_{k=1}^n|B_k(\cdot)u|^2,\; u\in V.\label{eqn:CoercivAss}
\end{eqnarray}
\item[\textbf{(AC3)}] There exists a function $\phi\in
L^{\infty}(0,T)$
such that % \Blue{The word a.s. should be
%restricted to statements related to $\mathbb{P}$.}
%\Red{The inequality \eqref{eqn:B_kWeakBoundAss}
%holds a.s. Can this set depend on $u$ or maybe we want it to the
%same for all $u$?  maybe separability of $V4$ could be helpful?}
%\Blue{I meant that this inequality is true everywhere (not a.s.)}
\begin{eqnarray}
\suml_{k=1}^n|\lb u,B_k(\cdot)u\rb|\leq \phi(\cdot)|u|^2,\;
\mbox{for all } t\in[0,T],\, u\in V.\label{eqn:B_kWeakBoundAss}
\end{eqnarray}
\item[\textbf{(AC4)}] There exist functions $K_1\in L^2(0,T)$ and
$K_2\in L^1(0,T)$ such that $K_2\geq 0$ and
\begin{eqnarray}\label{eqn:CWeakBoundAss}
C(t)=\suml_{k=1}^nB_k(t)^*[\tilde{A}(t),B_k(t)]\leq
K_1(t)\id+K_2(t)\tilde{A}(t),\; \mbox{for all } t\in [0,T],
\end{eqnarray}
where $\id$ is the identity operator.
\item[\textbf{(AC5)}] There exists constants $L_1,L_2>0$ such that
\begin{eqnarray}
\suml_k\Vert B_k(\cdot)x\Vert\leq L_1|A(\cdot)x|+L_2|x|,\; \mbox{
for all } t\in [0,T],\, x\in D(A).\label{eqn:B_strongboundAss}
\end{eqnarray}

\item[\textbf{(AC6)}] There exist constants  $\beta,\gamma>0$ such
that for every $t\in [0,T]$ we have
\begin{eqnarray}
|\lb A(\cdot)x,x\rb|\leq \beta\Vert x\Vert^2+\gamma|x|^2, x\in
V.\label{eqn:AStrongBoundAss}
\end{eqnarray}
\item[\textbf{(AC7)}] There exist functions $C_1^k\in
L^2(0,T),k=1,\ldots,n$ such that
\begin{equation}
|\lb \tilde{A}x,B_k(t)x\rb|\leq C_1^k(t)|\lb \tilde{A}x,x\rb|,
\mbox{ for $x\in V$ and a.a. }\,
t\in[0,T].\label{eqn:B_kFirstOrder}
\end{equation}
\end{trivlist}

\begin{theorem}\label{thm:BackwardUniqueness}
Let us assume  that maps $A:[0,T]\to \mathcal{L}(V,V^\prime)$ and
$B^k:[0,T]\to \mathcal{L}(V,H)$, $k=1,\ldots,n$  satisfy the
 assumptions \textbf{(AC0)}-\textbf{(AC6)}.

Assume that $u_0\in H$ and that a process $u$ satisfies  the
following  conditions.
 There
exist constants  $\delta_0>0$ and $\kappa>2+\frac{4}{\delta_0}$
such that
%\Red{I still believe that the first assumption below is redundant.
%We only need to assume continuity.}
\begin{eqnarray}
u&\in& L^{2+\delta_0}(\Omega,C([0,T],V)),\label{eqn:SolutionAss-2}
\\
u &\in & \mathcal{M}^2(0,T;D(\hat{A})).\label{eqn:SolutionAss-1'}
\end{eqnarray}
There exists a progressively measurable process $n$ such that
\begin{eqnarray}
&&\mathbb{E}e^{\kappa
\intl_{0}^Tn^2(s)\,ds}<\infty,\label{eqn:NonlinearityAss-n}
\\
|F(t,u(t))| &\leq & n(t)\Vert u(t)\Vert, \;  \mbox{ for a.a. }\,
t\in[0,T].\label{eqn:NonlinearityAss}
\end{eqnarray}
%

%\Red{What
%measurability of $n$ do we require? Predictability.} \Blue{What
%wrong would it be to assume only that $n$ is progressively
%measurable?}

%
Assume that $u$ is a solution of problem \eqref{eqn-1.1}.  If
$u(T)=0$ $\mathbb{P}$-a.s., then
 $u(t)=0$, $\mathbb{P}$-a.s. for all $t\in[0,T]$.
\end{theorem}

\begin{remark}\label{rem:rem-1} It is well known that under some appropriate assumptions on $F$
the problem \eqref{eqn-1.1} has a unique solution see e.g.
 Theorem 1.4, p.140 in  \cite{[Pardoux_1979]} for the case
 $F=0$.%\Red{More examples?}
\end{remark}
\begin{remark}
The Assumption \eqref{eqn:NonlinearityAss-n} is satisfied if, e.g.
$n\in L^2(0,T)$ is a deterministic function.
\end{remark}
The Theorem \ref{thm:BackwardUniqueness} is well adjusted for
linear parabolic SPDE or if nonlinearity has no more than "linear"
growth (see section \ref{sec:appl} for examples). Indeed, in this
case $n$ from the assumptions
\eqref{eqn:NonlinearityAss-n}--\eqref{eqn:NonlinearityAss} is
deterministic function. The following Theorem is more suited to
deal with parabolic SPDE with quadratic nonlinearity such as
stochastic NSE. This improvement is achieved at the "price" of
more stringent assumptions on the type of noise we consider.
\begin{theorem}\label{thm:NonlinearBackwardUniqueness}
Let us assume  that maps $A:[0,T]\to \mathcal{L}(V,V^\prime)$ and
$B^k:[0,T]\to \mathcal{L}(V,H)$, $k=1,\ldots,n$  satisfy the
 assumptions \textbf{(AC0)}-\textbf{(AC7)}. Furthermore, assumption \textbf{(AC4)} is satisfied with $K_1=0$.

Assume that $u_0\in H$ and that a process $u$ satisfies  the
following  conditions.
 There
exist a constant  $\delta_0>0$ such that
%\Red{I still believe that the first assumption below is redundant.
%We only need to assume continuity.}
\begin{eqnarray}
u&\in& L^{2+\delta_0}(\Omega,C(0,T;V)),\label{eqn:NSolutionAss-2}
\\
u &\in & \mathcal{M}^2(0,T;D(\hat{A}))\label{eqn:NSolutionAss-1'}.
\end{eqnarray}
There exists a progressively measurable process $n$ such that
\begin{eqnarray}
&& \intl_{0}^Tn^2(s)\,ds<\infty\mbox{
a.a.},\label{eqn:NNonlinearityAss-n}
\\
|F(t,u(t))| &\leq & n(t)\Vert u(t)\Vert, \;  \mbox{ for a.a. }\,
t\in[0,T].\label{eqn:NNonlinearityAss}
\end{eqnarray}

Assume that $u$ is a solution of problem \eqref{eqn-1.1}.  If
$u(T)=0$ $\mathbb{P}$-a.s., then
 $u(t)=0$, $\mathbb{P}$-a.s. for all $t\in[0,T]$.
\end{theorem}

\begin{corollary}\label{cor:cor-1}
Under the assumptions of Theorem \ref{thm:BackwardUniqueness} or
Theorem \ref{thm:NonlinearBackwardUniqueness} either $u(t)=0$,
$\mathbb{P}$-a.s. for all $t\in[0,T]$ or $|u(t)|>0$,
$\mathbb{P}$-a.s. for all $t\in[0,T]$.
\end{corollary}
We will use in the following Theorem the same notation as in
Theorem \ref{thm:NonlinearBackwardUniqueness}.

%\Red{Where do we use uniqueness? Maybe we should really assume that
%the equation (1.1) possesses uniqueness property?} \Green{Indeed,
%you are  right. We really need uniqueness. We use it in the end of
%the proof of step 1. We show that if solution equal $0$ at some
%moment $t_0$ then it is $0$ everywhere. Indeed by backward
%uniqueness solution should be $0$ at the initial point and at this
%point we use uniqueness of solution of equation to conclude that it
%is $0$ everywhere.}

\begin{theorem}\label{thm:ExistenceSpectrLimit}
Let us assume  that maps $A:[0,\infty)\to \mathcal{L}(V,V^\prime)$
and $B^k:[0,\infty)\to \mathcal{L}(V,H)$, $k=1,\ldots,n$  satisfy
the assumptions \textbf{(AC0)}-\textbf{(AC7)} on any finite
interval. Furthermore, we assume that the assumptions
\textbf{(AC1)}, \textbf{(AC3)}, \textbf{(AC4)} and \textbf{(AC7)}
are satisfied globally on $[0,\infty)$ and assumption
\textbf{(AC4)} is satisfied with parameter $K_1=0$.

Suppose that $u$ is a unique solution of problem \eqref{eqn-1.1}
with $u(0)\not=0$.  We also assume that
\begin{trivlist}
%\item
%\begin{equation}
%u\in L^2_{\rm loc}(T_0,\infty;D(\hat{A})))=1,\;
%\mathbb{P}-\mbox{a.s.}, \label{eqn:RegularityAssumption}
%\end{equation}
\item %and  that
there exist a progressively measurable process $n$ and a
measurable set $\Omega^\prime\subset \Omega$,
$\mathbb{P}(\Omega^\prime)=1$ such that for all
$\omega\in\Omega^\prime$,  $n(\cdot,\omega)\in L^2(T_0,\infty)$
and
\begin{equation}
|F(t,u(t))|\leq n(t)\Vert u(t)\Vert,\; \mbox{ for a.a.
}t\in[T_0,\infty).\label{eqn:NonlinearityAss-2}
\end{equation}
\item Assume that
\begin{eqnarray}
&&u \in \mathcal{M}^2_{\rm loc}(0,\infty;D(\hat{A})),\label{eqn:SpLSolutionAss-1}\\
\forall T>0 &&\mathbb{E}\supl_{0\leq t\leq
T}||u(t)||^{2+\delta_0}<\infty.
%L^{\infty}(\Omega,L^{\infty}([0,T],V)), \forall T>0.
\label{eqn:SpLSolutionAss-2}
\end{eqnarray}

\end{trivlist}
Then there exists a measurable map
$\tilde{\Lambda}^{\infty}:\Omega\to\sigma(\tilde{A})$ such that
$\mathbb{P}$-a.s..
$$
\liml_{t\to\infty}\frac{\lb \tilde{A}u(t),u(t)
\rb}{|u(t)|^2}=\tilde{\Lambda}^{\infty}.
$$
\end{theorem}

\section{Proof of the Theorems \ref{thm:BackwardUniqueness} and \ref{thm:NonlinearBackwardUniqueness} on the  backward uniqueness for SPDEs}
\label{sec:proof_thm_backward}

\begin{proof}[Proof of Theorem \ref{thm:BackwardUniqueness}]
We will argue by contradiction. Suppose that the assertion of the
Theorem is not true. Then, because the process $u$ is adapted, we
will be able to find   $t_0\in[0,T)$, an event
$R\in\mathcal{F}_{t_0} $ and  a constant $c>0$ such that
$\mathbb{P}(R)>0$ and
\begin{equation}
|u(t_0,\omega)|\geq c>0,\; \omega\in R.\label{eqn:LowerBoundSol}
\end{equation}

Without loss of generality we can  assume  that $\mathbb{P}(R)=1$
and $t_0=0$. Otherwise, we can consider instead of measure
$\mathbb{P}$ the conditional measure
$\mathbb{P}_R:=\frac{\mathbb{P}(\cdot\cap R)}{\mathbb{P}(R)}$.

%Since $u(\cdot,\omega)\in C([t_0,T],H)$, $\mathbb{P}$-a.s., we have the following alternative.
%\Red{I have changed (i) and (ii) below.}
%\begin{trivlist}
%\item[(i)] \textit{Either} for each $t\in (t_0,T]$ and each measurable set  $A\subset R$ with $\mathbb{P}(A)>0$ such that $|u(t)|>0$ for all $\omega\in A$,
%$\mathbb{P}$-a.s. \item[(ii)] \textit{or} there exist
%$\tau\in (t_0, T]$ and a measurable set  $A\subset R$ with $\mathbb{P}(A)>0$ such that
%$|u(\tau,\omega)|=0$ for $\omega\in A$.

%$t_0<\tau\leq T$, \del{is correctly defined} where $\tau(\omega)=\inf\{t\in
%(t_0,T],|u(t)|=0\}$.
%\end{trivlist}
%In the first case we get  an obvious contradiction. Let us
%consider now the second case. We can always assume that
%$\mathbb{P}(A)=1$.
Suppose  that there exists a constant $c>0$ and  a probability
measure\footnote{Here and below by $\mathbb{E}_{\mathbb{Q}}$ we
will denote the mathematical expectation w.r.t. the measure
$\mathbb{Q}$.} $\mathbb{Q}$ equivalent to $\mathbb{P}$ such that
\del{Corrected.I see a problem here as $\tau$ is a random time.}
$$\mathbb{E}_{\mathbb{Q}}|u(t)|^2\geq c, \; \mbox{for all } t\in [0,T].$$
Then, by taking $t=T$, we infer that
$\mathbb{E}_{\mathbb{Q}}|u(T)|^2>0$ what is a clear contradiction
with the assumption that  $u(T)=0$,  $\mathbb{P}$-a.s..

Now we shall  prove that such a measure  exists.
%For this
%let us fix $\eps>0$ and  let us define a process $\mathcal{M}^\eps=(\mathcal{M}^\eps_t)_{t\in [0,T]}$ by
%\begin{equation}
%M_\eps(t)=\exp(-2\suml_k\intl_0^t\frac{\lb u,B_ku\rb}{|u|^2+\eps}\, dw_s^k-
%2\intl_0^t\suml_k\frac{(\lb
%u,B_ku\rb)^2}{(|u|^2+\eps)^2}\,ds),\; t\in [0,T].
%\label{eqn:ExpMart-1}
%\end{equation}
%We have
%\begin{equation}
%dM_\eps(t)=-2M_\eps(t)\suml_k\frac{\lb
%u,B_ku\rb}{|u|^2+\eps}\, dw_s^k, \; t\in
%[0,T].\label{eqn:M_t}
%\end{equation}
%Indeed, since the process $u$ is progressively measurable we have
%also that process $\rho=\frac{\lb u,B_ku\rb}{|u|^2+\eps}$ is
%progressively measurable. Furthermore, because of the assumption
%\eqref{eqn:B_kWeakBoundAss} we have that $\rho\leq \phi\in
%L^2(0,T)$. Therefore, $\intl_0^{\cdot}\rho_s\, dw(s)$ is square
%integrable martingale and condition 5.7 of Theorem 5.3 p.142 in
%\cite{[Ikeda-1981]} is satisfied and hence  $\mathbb{E}M_\eps(t)=1$ for all $t\geq 0$, and thus $M_\eps$ is a continuous martingale. %and
%$$
%dM_\eps(t)=-2M_\eps(t)\suml_k\frac{\lb u,B_ku\rb}{|u|^2+\eps}\, dw_s^k, \; t\in [0,T].
%$$
%Indeed, \Red{How do you get the inequality below?}
%$$
%\mathbb{E}|M_\eps(t)|^{2}\leq \mathbb{E}
%e^{4\intl_0^t\suml_k\frac{(\lb
%u,B_ku\rb)^2}{(|u|^2+\eps)^2}\,ds}\leq
%e^{4\intl_0^T|\phi(s)|^2\, ds}<\infty,
%$$
%and the result follows from Theorems 5.2 and 5.3, p.142 in
%\cite{[Ikeda-1981]}.

For this let us fix $\delta\geq 0$, $k=1,\cdots,n$ and let us
define progressively measurable processes
$$\rho^\delta_k(s)=\frac{\lb
u(s),B_ku(s)\rb}{|u(s)|^2+\delta},\; s\in [0,T].$$ Because of the
assumption \eqref{eqn:B_kWeakBoundAss}, $\vert \rho^\delta_k
\vert\leq \phi$ and since $\phi\in L^\infty(0,T) \subset L^2(0,T)$
we infer that $\rho^\delta_k\in \mathcal{M}(0,T;\mathbb{R})$.
Therefore, a process
$$\sum_{k=1}^n\intl_{0}^{t}\rho^\delta_k(s)\, dw^k_s, \; t\in [0,T],$$ is square
integrable $\mathbb{R}$-valued martingale which satisfies
condition 5.7 from \cite[Theorem 5.3 p.142]{[Ikeda-1981]}. Hence
the process $M_\delta=(M_\delta(t))_{t\in [0,T]}$ defined by
\begin{equation}
M_\delta(t)=\exp(-2\suml_k\intl_{0}^t \rho^\delta_k(s) \, dw^k(s)-
2\intl_{0}^t\suml_k |\rho^\delta_k(s)|^2\,ds),\; t\in [0,T],
\label{eqn:ExpMart-1}
\end{equation}
satisfies $\mathbb{E}M_\delta(t)=1$ for all $t\geq 0$ and
\begin{equation}
dM_\delta(t)=-2M_\delta(t)\suml_k\frac{\lb
u,B_ku\rb}{|u|^2+\delta}\, dw^k(t), \; t\in [0,T].\label{eqn:M_t}
\end{equation}
%\Red{Logical train here seems to be wrong. First of all,
%\cite{[Ikeda-1981]} assume that $M_\delta$ belongs to a class
%denoted there by $\mathcal{M}$. Have we verified this? Secondly,
%it is part of Theorem 5.3 that $M_\delta$ is a continuous
%martingale.}

Therefore, $M_\delta$ is a continuous square integrable %\Red{How
%do we actually prove this property?}
martingale. The above allows us to define a probability
measure\footnote{Note that the measure $\mathbb{Q}^0$ is also well
defined by this formula.} $\mathbb{Q}^{\delta}$ by
$$\frac{d\mathbb{Q}^{\delta}}{d\mathbb{P}}=M_\delta(T).$$

Next let us fix $\eps>0$ and  define a process $\psi^{\eps}$ by
\begin{equation}
\psi^{\eps}(t)=-\frac{1}{2}M_\eps(t)\log{(|u(t)|^2+\eps)},\;
t\in[0,T].\label{eqn:Psi_Def_1}
\end{equation}
\del{When we write $dx(t)=f(t)\,dw(t)$, we can have two cases. 1.
$\mathbb{E}\int_0^T |f(t)|^2\, dt<\infty$ or $\int_0^T |f(t)|^2\,
dt<\infty$ with $\mathbb{P}=1$. Which case do we have here? First
one.}

Now we will prove the following result.
\begin{lemma}\label{lem-psi_ito}
The process  $\psi^{\eps}$ defined in \eqref{eqn:Psi_Def_1} is an
It\^o process.
\end{lemma}

\begin{proof}[Proof of Lemma \ref{lem-psi_ito}]
First,  by  invoking   Theorem 1.2 from Pardoux
\cite{[Pardoux_1979]}, we will show that $\log{(|u(t)|^2+\eps)}$
is an It\^o process. For this we need to check that the
assumptions of that result for the function $R:H\ni x\mapsto \log
(|x|^2+\eps)\in\mathbb{R}$ and the process $u$ are satisfied.
Obviously the function $R$ is of $C^2$-class and since
\begin{eqnarray*}
R^\prime(x)h& = &\frac{2\lb x,h\rb}{|x|^2+\eps}, \;x,h\in H,\\
R^{\prime\prime}(x)(h_1,h_2)&=&\frac{2}{|x|^2+\eps}(\lb
h_2,h_1\rb-\frac{2\lb x,h_2\rb\lb x,h_1\rb}{|x|^2+\eps}),\;
x,h_1,h_2\in H,
\end{eqnarray*}
the $1^{\rm st}$ and $2^{\rm nd}$ derivatives of $R$ are bounded
and hence the assumptions (i) and (ii) of \cite[Theorem
1.2]{[Pardoux_1979]} are satisfied. Since the embedding $V\embed
H$ is continuous, we  infer that the assumptions (iv) and (v) of
\cite[Theorem 1.2]{[Pardoux_1979]} are satisfied as well.
Moreover, for any $Q\in \mathcal{T}_1(H)$, we have
\begin{eqnarray*}
Tr(Q\circ R^{\prime\prime}(x)) &=& \frac{2}{|x|^2+\eps}( Tr
Q-\frac{2\lb Qx,x\rb}{|x|^2+\eps}),\; x\in H
\end{eqnarray*}
and hence  the map $H\ni x\mapsto Tr(Q\circ
R^{\prime\prime})(x)\in\mathbb{R}$ is continuous.  Thus also the
condition (iii) in    \cite[Theorem 1.4]{[Pardoux_1979]} is
satisfied. Therefore, $\log{(|u(t)|^2+\eps)}$, $t\geq 0$,  is an
It\^o process and
%\Red{why?Changed}
\begin{eqnarray}
-d\frac12\log{(|u(s)|^2+\eps)})&=&\suml_{k=1}^n\frac{\lb u,B_ku\rb }{|u|^2+\eps}\, dw^k(s) \nonumber\\
+\, \Big(\suml_{k=1}^n \frac{\lb u,B_ku\rb^2}{(|u|^2+\eps)^2} &+&
\frac{\lb
(A-\frac{1}{2}\suml_{k=1}^nB_k^*B_k)u+F(s,u),u\rb}{(|u|^2+\eps)^2}\Big)\,ds.\label{eqn:Aux-2}
\end{eqnarray}
Secondly, since $M_\eps(t)$, $t\geq 0$,  is a continuous square
integrable martingale satisfying equality \eqref{eqn:M_t},   the
process
$\psi^{\eps}(t)=-\frac{1}{2}M_\eps(t)\log{(|u(t)|^2+\eps)}$, $t\in
[0,T]$, is an It\^o process and
\begin{eqnarray}
d\psi^{\eps}(t)&=&-\frac{1}{2}\big(\log{(|u|^2+\eps)}dM_\eps(t)+M_\eps d\log{(|u|^2+\eps)}\nonumber\\
&+&d\lb M_\eps,\log{(|u|^2+\eps)}\rb_t\big).\label{eqn:Aux-3}
\end{eqnarray}
This completes the proof the Lemma.
\end{proof}

Let us define functions $\widetilde{\Lambda}_{\eps}$ and
$\widetilde{\Lambda}_{\eps}^F$ by
\begin{eqnarray}
\widetilde{\Lambda}_{\eps}(u)&=&\frac{\lb (A-\frac{1}{2}
\suml_{k=1}^nB_k^*B_k)u,u\rb}{|u|^2+\eps}+\suml_{k=1}^n\frac{\lb u,B_ku\rb^2}{(|u|^2+\eps)^2},\, u\in V, \label{eqn:LA_eps_def}\\
\widetilde{\Lambda}_{\eps}^F(t,u)&=&\widetilde{\Lambda}_{\eps}(u)+
\frac{\lb F(t,u),u\rb}{|u|^2+\eps},\; t\in [0,T],\, u\in
V,\eps\geq 0.\label{eqn:LA_eps_F_def}.
\end{eqnarray}
We will omit index $\eps$ if $\eps=0$.

Combining equalities \eqref{eqn:M_t}, \eqref{eqn:Aux-2} and
\eqref{eqn:Aux-3} we infer that
\begin{multline}\label{eqn:DerivPsi}
d\psi^{\eps}(s)=M_\eps(s)\Big( \widetilde{\Lambda}_{\eps}^F(s,u(s))\,ds+\suml_{k=1}^n\frac{\lb u,B_ku\rb}{|u|^2+\eps}\, dw^k(s)\\
+\log{(|u|^2+\eps)}\suml_{k=1}^n\frac{\lb u,B_ku\rb\,
dw^k(s)}{|u|^2+\eps}+\suml_{k=1}^n\frac{\lb
u,B_ku\rb^2}{(|u|^2+\eps)^2}\,ds\Big).
\end{multline}
It follows from assumption \eqref{eqn:B_kWeakBoundAss} that the
integrands in the stochastic integrals in \eqref{eqn:DerivPsi}
belong to $\mathcal{M}^2(0,T;\mathbb{R})$. Therefore, we can apply
mathematical expectation to \eqref{eqn:DerivPsi} and consequently,
by  Assumption \eqref{eqn:B_kWeakBoundAss} we infer that for all
$t\in [0,T]$,
\begin{equation}
\label{eqn:Part1AuxInequality-1}
\frac{1}{2}\mathbb{E}M_\eps(0)\log{(|u(0)|^2+\eps)}-\frac{1}{2}\mathbb{E}M_\eps(t)\log{(|u(t)|^2+\eps)}\leq
C_1+C_2\intl_{0}^t\mathbb{E}M_\eps(s)\widetilde{\Lambda}_{\eps}^F(s,u(s))\,ds.
\end{equation}

Suppose now that the following result is true.
\begin{lemma}\label{lem_est}
In the above framework we have
$$\supl_{\eps>0}\intl_{0}^T \mathbb{E}M_\eps(s)\widetilde{\Lambda}_{\eps}^F(s,u(s))\,ds = \supl_{\eps>0}\intl_{0}^T \mathbb{E}_{\mathbb{Q}^\eps}\widetilde{\Lambda}_{\eps}^F(s,u(s))\, ds<\infty.$$
\end{lemma} Then,  in conjunction with \eqref{eqn:Part1AuxInequality-1} and
\eqref{eqn:Part1AuxInequality-2}  we have
 \del{$K=\intl_{0}^T|C(s)|\,ds < \infty$ we will have} that %\Red{I cannot see the below!}
\begin{eqnarray*}
\mathbb{E}_{\mathbb{Q}^\eps}\log{(|u(t)|^2+\eps)}&\geq &
\mathbb{E}_{\mathbb{Q}^\eps}\log{(|u(0)|^2+\eps)}-K\\
&=& \mathbb{E}\log{(|u(0)|^2+\eps)}-K,\; t\in [0,T],
\end{eqnarray*}
where $K=2C_1+2C_2\supl_{\eps>0}\intl_{0}^T
\mathbb{E}_{\mathbb{Q}^\eps}\widetilde{\Lambda}_{\eps}^F(s,u(s))\,
ds$. Hence  by the Jensen inequality
\begin{equation}
\mathbb{E}_{\mathbb{Q}^\eps}(|u(t)|^2+\eps) =
\mathbb{E}_{\mathbb{Q}^\eps}e^{\log{(|u(t)|^2+\eps)}}
%\nonumber\\
\geq  e^{\mathbb{E}_{\mathbb{Q}^\eps}\log{(|u(t)|^2+\eps)}} \geq
e^{\mathbb{E}\log{(|u(0)|^2+\eps)}-K}\del{>0}.\label{eqn:AuxEstimate-1}
\end{equation}
Therefore, by the Fatou Lemma and \eqref{eqn:LowerBoundSol}, we
%\Red{Again, I do not see how the below line follows form the
%above!}
infer that
\begin{eqnarray}
\liminfl_{\eps\to 0}e^{\mathbb{E}\log{(|u(0)|^2+\eps)}-K} \geq
e^{\mathbb{E}\log{(|u(0)|^2)}-K}>0.\label{eqn:AuxEstimate-2}
\end{eqnarray}
Combining  \eqref{eqn:AuxEstimate-1} and
\eqref{eqn:Part1AuxInequality-2} we get %\Red{I am again confused.}
\begin{eqnarray}
\eps+\mathbb{E}[M_\eps(t)|u(t)|^2]=\mathbb{E}_{\mathbb{Q}^\eps}(|u(t)|^2+\eps)\geq
e^{\mathbb{E}\log{(|u(0)|^2)}-K}>0\label{eqn:AuxEstimate-2'}
\end{eqnarray}
%Now let us show that we can tend $\eps\to 0$ in the inequality
%\eqref{eqn:AuxEstimate-2'}. It is enough to prove that family
%$\{M_\eps(t)|u(t)|^2\}_{\eps>0}$ is uniformly integrable.
%Then, the result will follow from Theorem
%\ref{thm:LimitBehaviourUniformInt}. We have by the H{\"o}lder
%inequality with $p=q/(q-1)$, $q=(2+\kappa_0)/(2(1+\tau))$ (see
%assumption \eqref{eqn:SolutionAss-1} for definition of $\kappa_0$)
%$$
%\mathbb{E}|M_\eps(t)|u(t)|^2|^{1+\tau}\leq
%(\mathbb{E}|M_\eps(t)|^{p(1+\tau)})^{1/p}(\mathbb{E}|u(t)|^{2+\kappa_0})^{1/q}<\infty,\tau\in
%(0,\kappa_0/2)
%$$
%Indeed, we have $\mathbb{E}|M_\eps(t)|^{p(1+\tau)}<\infty$ by
%assumption $\eqref{eqn:B_kWeakBoundAss}$ and Theorem 5.3, p.142 of
%\cite{[Ikeda-1981]}. Also,
%$\supl_t\mathbb{E}|u(t)|^{2+\kappa_0}<\infty$ by assumption
%\eqref{eqn:SolutionAss-1}. Thus, by Theorem
%\ref{thm:UniformIntegrabilityCrit} family
%$\{M_\eps(t)|u(t)|^2\}_{\eps>0}$ is uniformly integrable and
%it follows from Theorem \ref{thm:LimitBehaviourUniformInt} that we
%can tend $\eps\to 0$ in the left part of inequality
%\eqref{eqn:AuxEstimate-2'}.
Choose $\eps=\frac{1}{2}e^{\mathbb{E}\log{(|u(0)|^2)}-K}$. Since
$u$
 is an $\mathbb{F}$-adapted process we get by
\eqref{eqn:Part1AuxInequality-2} and \eqref{eqn:AuxEstimate-2'}
\begin{equation}
\mathbb{E}_{\mathbb{Q}^{\eps}}|u(t)|^2=\mathbb{E}[M_\eps(t)|u(t)|^2]\geq
\frac{1}{2}e^{\mathbb{E}\log{(|u(0)|^2)}-K}>0,\label{eqn:AuxEstimate-3}
\end{equation}
what contradicts our assumption that $u(t)=0$ and the Theorem
follows. \end{proof}

Hence, it only remains to prove Lemma \ref{lem_est}.%, i.e. find an
%estimate on
%$\supl_{\eps>0}\intl_{0}^t\mathbb{E}_{\mathbb{Q}^\eps}\widetilde{\Lambda}_{\eps}^F(s,u(s))\,ds$.
The proof of Lemma \ref{lem_est} will be preceded by the following
auxiliary results:
\begin{lemma}\label{lem:AuxRes_2}
%Denote
%\begin{eqnarray}
%K_3(\eps,N)&=&\intl_{0}^tM_\eps(s)\frac{|(\tilde{A}_N-\tilde{A})u|^2}{|u|^2+\eps}\,ds,\\
%K_4(\eps,N)&=&\intl_{0}^tM_\eps(s)\frac{\suml_k\Vert
%Q_NB_ku\Vert^2}{|u|^2+\eps}\,ds.
%
In the above framework there exists set
$\Omega^\prime\subset\Omega$, $\mathbb{P}(\Omega^\prime)=1$ and
sequence $\{N_l\}_{l=1}^{\infty}$ such that for all
$\omega\in\Omega^\prime$ , $t\in[0,T]$, as $l\to\infty$, the
following holds
\begin{eqnarray}
&&\intl_{0}^tM_\eps(s)\frac{|(\tilde{A}_{l}-\tilde{A})u|^2}{|u|^2+\eps}\,ds\to 0,\label{eqn:AuxConv-1.1}\\
&&\intl_{0}^tM_\eps(s)\frac{\suml_{k=1}^n\Vert
Q_{l}B_ku\Vert^2}{|u|^2+\eps}\,ds\to 0,\label{eqn:AuxConv-2.1}\\
&&\frac{<(\tilde{A}_{N_l}-\tilde{A})u(t),u(t)>}{|u(t)|^2+\eps}\to
0,\label{eqn:AuxConv-3.1}\\
&&\intl_0^tM_\eps(s)\suml_{k=1}^n\frac{<(\tilde{A}_{N_l}-\tilde{A})u(s),B_ku(s)>}{|u(s)|_H^2+\eps}dw^k(s)\to
0.\label{eqn:AuxConv-4.1}
\end{eqnarray}
\end{lemma}
\begin{proof}[Proof of Lemma \ref{lem:AuxRes_2}]
It is enough to prove that for each of convergences
\eqref{eqn:AuxConv-1.1}-\eqref{eqn:AuxConv-4.1} we can find
$\Omega_i'\subset\Omega$, $i=1,\ldots,4$ of measure $1$ such that
the corresponding convergence holds. Denote $v_N=P_Nv$, $v\in H$.
Then, for $t\in[0,T]$, the following inequalities holds
\begin{equation}
\intl_{0}^tM_\eps(s)\frac{|(\tilde{A}_{N}-\tilde{A})u|^2}{|u|^2+\eps}\,ds\leq
\intl_{0}^TM_\eps(s)\frac{|(\tilde{A}_{N}-\tilde{A})u|^2}{|u|^2+\eps}\,ds
\leq
\frac{\supl_{s\in[0,T]}M_\eps(s)}{\eps}\intl_{0}^T|(\tilde{A}_{N}-\tilde{A})u|^2\,ds,\label{eqn:AuxConv-1.2a}
\end{equation}
and
\begin{eqnarray}
\intl_{0}^T|(\tilde{A}_{N}-\tilde{A})u|^2\,ds \leq
\Big(\intl_{0}^T|P_N^*\tilde{A}(P_Nu-u)|^2\,ds+
\intl_{0}^T|(P_N^*-\id)\tilde{A}u|^2ds\Big)\nonumber\\
\leq
\Big(|P_Nu-u|_{L^2(0,T;D(\hat{A}))}^2+|Q_N\tilde{A}u|_{L^2(0,T;H)}^2\Big).\label{eqn:AuxConv-1.2}
\end{eqnarray}
Since $M_{\eps}$ is a square integrable martingale, then by Doob
inequality there exist a set $\Omega_1\subset\Omega$ of measure
$1$ such that $\supl_{s\in[0,T]}M_\eps(s,\omega)<\infty$,
$\omega\in \Omega_1$. Moreover by Assumption
\eqref{eqn:SolutionAss-1'} there exist $\Omega_2\subset\Omega$ of
measure $1$ such that
\begin{equation}
u(\cdot,\omega)\in L^2(0,T;D(\hat{A})),\;\omega\in
\Omega_2.\label{eqn:AuxAssumption}
\end{equation}
Therefore, if $\omega\in \Omega_1'=\Omega_1\cap\Omega_2$, then by
the Lebesgue Dominated Convergence Theorem,
\begin{equation}
\frac{2\supl_{s\in[0,T]}M_\eps(s,\omega)}{\eps}\Big(|P_Nu(\omega)-u(\omega)|_{L^2(0,T;D(\hat{A}))}^2+|Q_N\tilde{A}u(\omega)|_{L^2(0,T;H)}^2\Big)\to
0,N\to\infty.\label{eqn:AuxConv-1.3}
\end{equation}
Consequently, by inequalities \eqref{eqn:AuxConv-1.2a} and
\eqref{eqn:AuxConv-1.2} we infer that
\begin{equation}
\intl_{0}^tM_\eps(s)\frac{|(\tilde{A}_{N}-\tilde{A})u|^2}{|u|^2+\eps}\,ds\to
0,N\to\infty,\omega\in \Omega_1\cap\Omega_2, t\in
[0,T].\label{eqn:AuxConv-1.4}
\end{equation}
This concludes the proof of \eqref{eqn:AuxConv-1.1}.

Similarly to \eqref{eqn:AuxConv-1.2a} we get
\begin{equation}
\intl_{0}^tM_\eps(s)\frac{\suml_k\Vert
Q_{N}B_ku\Vert^2}{|u|^2+\eps}\,ds\leq
\frac{\supl_{s\in[0,T]}M_\eps(s,\omega)}{\eps}\intl_{0}^T\suml_{k=1}^n\Vert
Q_{N}B_ku\Vert^2\,ds.\label{eqn:AuxConv-2.2}
\end{equation}
If $\omega\in \Omega_2$ then by assumptions
\eqref{eqn:B_strongboundAss} and \eqref{eqn:AuxAssumption} that
$B_ku(\omega)\in L^2(0,T;V),k=1,\ldots,n$. Hence, by the Lebesgue
Dominated Convergence Theorem for $\omega\in
\Omega_1\cap\Omega_2$,
\begin{equation}
\frac{\supl_{s\in[0,T]}M_\eps(s,\omega)}{\eps}\intl_{0}^T\suml_{k=1}^n\Vert
Q_{N}B_ku\Vert^2\,ds\to 0,\mbox{ as
}N\to\infty.\label{eqn:AuxConv-2.3}
\end{equation}
Therefore, by combining \eqref{eqn:AuxConv-2.2} and
\eqref{eqn:AuxConv-2.3}, we infer that
\begin{equation}
\intl_{0}^tM_\eps(s)\frac{\suml_k\Vert
Q_{N}B_ku\Vert^2}{|u|^2+\eps}\,ds\to 0,N\to\infty,\omega\in
\Omega_1\cap\Omega_2,t\in[0,T].\label{eqn:AuxConv-2.4}
\end{equation}
This proves \eqref{eqn:AuxConv-2.1} with $\Omega_2'=\Omega_1'$.

From Assumption \eqref{eqn:SolutionAss-2} by the Lebesgue
Dominated Convergence Theorem and Dini's Theorem we infer that
\begin{equation}
\mathbb{E}\supl_{s\in[0,T]}||u_m(s)-u(s)||_V^{1+\delta_0}\to
0,m\to\infty
\end{equation}
Therefore, by Chebyshev inequality,
$\supl_{s\in[0,T]}||u_m(s)-u(s)||_V^{1+\delta_0}$ converges to $0$
as $m\to\infty$ in probability. Consequently, there exist sequence
$\{m_k\}_{k=1}^{\infty}$
%such that
%\begin{equation}
%\mathbb{E}\supl_{s\in[0,T]}||u_{m_k}(s)-u(s)||_V^{1+\delta_0}\leq
%\frac{1}{k^2},k\in\mathbb{N}.
%\end{equation}
%Consequently, by Fubini Theorem
%\begin{equation}
%\mathbb{E}\suml_{k=1}^{\infty}\supl_{s\in[0,T]}||u_{m_k}(s)-u(s)||_V^{1+\delta_0}=\suml_{k=1}^{\infty}\mathbb{E}\supl_{s\in[0,T]}||u_{m_k}(s)-u(s)||_V^{1+\delta_0}\leq
%\suml_{k=1}^{\infty}\frac{1}{k^2}<\infty.\label{eqn:AuxConv-3.2}
%\end{equation}
%By Borel-Cantelli Lemma we infer from inequality
%\eqref{eqn:AuxConv-3.2} that there exist
and set $\Omega_3\subset\Omega$, $\mathbb{P}(\Omega_3)=1$ such
that
\begin{equation}
\supl_{s\in[0,T]}||u_{m_k}(s,\omega)-u(s,\omega)||_V\to
0,k\to\infty,\omega\in \Omega_3.\label{eqn:AuxConv-3.3}
\end{equation}
Now we shall prove that \eqref{eqn:AuxConv-3.1} holds with the
subsequence $\{m_l\}_{l=1}^{\infty}$ defined above and $\omega\in
\Omega_3$. The following sequence of inequalities holds for $v\in
V, t\in[0,T]$
\begin{multline}
\frac{|<(\tilde{A}_{m}(t)-\tilde{A}(t))v,v>|}{|v|^2+\eps}\leq
\frac{1}{\eps}\left(|<\tilde{A}(t)v_m,v_m>-<\tilde{A}(t)v,v>|\right)\\\leq
\frac{1}{\eps}\left(|<\tilde{A}(t)v_m,v_m>-<\tilde{A}(t)v_m,v>|+|<\tilde{A}(t)v_m,v>-<\tilde{A}(t)v,v>|\right)\\
\leq\frac{1}{\eps}\left(|<\tilde{A}(t)v_m,v_m-v>|+|<\tilde{A}(t)(v_m-v),v>|\right)\\
\leq\frac{|\tilde{A}(t)|_{\mathcal{L}(V,V')}}{\eps}(||v_m||_V||v_m-v||_V+||v||_V||v_m-v||_V)
\leq\frac{2|\tilde{A}(t)|_{\mathcal{L}(V,V')}}{\eps}||v||_V||v_m-v||_V
\end{multline}
Therefore, by \eqref{eqn:AuxConv-3.3}, Assumption
\eqref{eqn:SolutionAss-2} and Assumption \textbf{(AC0)} we infer
that
\begin{equation}\label{eqn:AuxConv-3.4}
\liml_{l\to\infty}\supl_{s\in[0,T]}\frac{|<(\tilde{A}_{m_l}(s)-\tilde{A}(s))u(s),u(s)>|}{|u(s)|^2+\eps}=0,\omega\in\Omega_3.
\end{equation}
This proves \eqref{eqn:AuxConv-3.1}.
%Let us notice that the expression
%$\frac{|(\tilde{A}_{N}-\tilde{A})u|^2}{|u|^2+\eps}$ appearing in
%\eqref{eqn:AuxConv-1.1} can be bounded from above by
%$\frac{|\tilde{A}u|^2}{\eps}$.
%
%
%
%
%
%Therefore, by a simple application of H\"older inequality
%
%We have following estimates
%\begin{eqnarray}
%\mathbb{E}\intl_{0}^tM_\eps(s)\frac{|(\tilde{A}_{N}-\tilde{A})u|^2}{|u|^2+\eps}\,ds\leq
%
%
%\end{eqnarray}

It remains to prove convergence in \eqref{eqn:AuxConv-4.1}.
%Denote
%$$
%\delta=\frac{\delta_0}{2(4+\delta_0)},p_1=4+\frac{8}{\delta_0},p_2=\frac{2}{1+\delta},p_3=\frac{2+\delta_0}{1+\delta}.
%$$
By Assumption \eqref{eqn:SolutionAss-1'} and the Lebesgue
Dominated Convergence Theorem there exist a subsequence
$\{M_l\}_{l=1}^{\infty}$ of the sequence $\{m_k\}_{k=1}^{\infty}$
such that
\begin{equation}
|u_{M_l}-u|_{\mathcal{M}^2(0,T;D(\hat{A}))}+|P_{M_l}\tilde{A}(\cdot)u(\cdot)-\tilde{A}(\cdot)u(\cdot)|_{\mathcal{M}^2(0,T;H)}\leq\frac{1}{l}.\label{eqn:AuxConv-4.2}
\end{equation}
Denote
$$
p_1=2+\frac{4}{\delta_0},p_2=2+\delta_0, p_3=2.
$$
Let us observe that $\frac{1}{p_1}+\frac{1}{p_2}+\frac{1}{p_3}=1$.
Therefore, by H{\"o}lder and Burkholder inequalities, see e.g.
Corollary iv.4.2 in \cite{[RevuzYor-1999]}, we have
\begin{eqnarray}
&&\hspace{-2truecm}\lefteqn{\mathbb{E}\left|\supl_{t\in[0,T]}\intl_0^tM_\eps(s)
\suml_{k=1}^n\frac{<(\tilde{A}_{M_l}-\tilde{A})u(s),B_ku(s)>}{|u(s)|_H^2+\eps}dw^k(s)\right|}
\nonumber\\
&&\leq
C\mathbb{E}\left(\intl_0^TM_\eps^2(s)\suml_{k=1}^n\frac{|<(\tilde{A}_{M_l}-\tilde{A})u(s),B_ku(s)>|^2}{(|u(s)|_H^2+\eps)^2}ds\right)^{1/2}\nonumber\\
&&\leq\frac{C}{\eps}\mathbb{E}\left[\supl_{s\in[0,T]}\left(M_\eps(s)\suml_k|B_ku(s)|^2\right)\left(\intl_0^T|(\tilde{A}_{M_l}-\tilde{A})u(s)|^2\,ds\right)^{1/2}\right]\nonumber\\
&&\leq\frac{C}{\eps}\left(\mathbb{E}\supl_{s\in[0,T]}M_\eps^{p_1}(s)\right)^{1/p_1}
\left(\mathbb{E}\supl_{s\in[0,T]}\left(\suml_{k=1}^n(|B_ku|^2(s))\right)^{p_2/2}\right)^{1/p_2}\nonumber\\
&&\hspace{0.5truecm}\lefteqn{\times\left(\mathbb{E}\left(\intl_0^T|(\tilde{A}_{M_l}-\tilde{A})u|^2(s)ds\right)^{p_3/2}\right)^{1/p_3}}
\nonumber\\
&&=\frac{C}{\eps}\left(\mathbb{E}\supl_{s\in[0,T]}M_\eps^{p_1}(s)\right)^{1/p_1}
\left(\mathbb{E}\supl_{s\in[0,T]}\left(\suml_{k=1}^n(|B_ku|^2(s))\right)^{(2+\delta_0)/2}\right)^{1/p_2}\nonumber\\
&&\times\left(\mathbb{E}\intl_0^T|(\tilde{A}_{M_l}-\tilde{A})u|^2(s)ds\right)^{1/2}\label{eqn:AuxConv-4.3}
\end{eqnarray}
Next we will show that the RHS of \eqref{eqn:AuxConv-4.3} is
finite. Notice that the first factor on the RHS of
\eqref{eqn:AuxConv-4.3} is finite by the Doob inequality and
Assumption \eqref{eqn:B_kWeakBoundAss}. Furthermore, the second
term is finite by assumptions
\eqref{eqn:CoercivAss},\eqref{eqn:AStrongBoundAss} and
\eqref{eqn:SolutionAss-2}. Now we will find the upper bound for
the last term in the product $W$. We have
\begin{eqnarray}
&&\intl_0^T\mathbb{E}|(\tilde{A}_{M_l}-\tilde{A})u|^2(s)ds\nonumber\\
&\leq &
C\left(\intl_0^T\mathbb{E}|P_{M_l}^*\tilde{A}P_{M_l}u-P_{M_l}^*\tilde{A}u|^2(s)ds+
\intl_0^T\mathbb{E}|P_{M_l}^*\tilde{A}u-\tilde{A}u|^2(s)ds\right)\nonumber\\
&\leq &
C\left(\intl_0^T\mathbb{E}|\tilde{A}P_{M_l}u-\tilde{A}u|^2(s)ds+
\intl_0^T\mathbb{E}|Q_{M_l}\tilde{A}u|^2(s)ds\right)\nonumber\\
&&=C\left(|u_{M_l}-u|_{\mathcal{M}^2(0,T;D(\hat{A}))}^2
+|Q_{M_l}\tilde{A}u|_{\mathcal{M}^2(0,T;H)}^2\right)\leq
\frac{C}{l^2}\label{eqn:AuxConv-4.4}
\end{eqnarray}
where last inequality follows from assumption
\eqref{eqn:AuxConv-4.2}. Combining \eqref{eqn:AuxConv-4.3} and
\eqref{eqn:AuxConv-4.4} we infer that
\begin{equation}
\mathbb{E}\supl_{t\in[0,T]}
\left|\intl_0^tM_\eps(s)\suml_{k=1}^n\frac{<(\tilde{A}_{M_l}-\tilde{A})u(s),B_ku(s)>}{|u(s)|_H^2+\eps}dw^k(s)\right|
\leq \frac{C}{l^2}.\label{eqn:AuxConv-4.5}
\end{equation}
By Borel-Cantelli Lemma and Doob inequality we infer from
inequality \eqref{eqn:AuxConv-4.5} that there exist $\Omega_4$,
$\mathbb{P}(\Omega_4)=1$ such that
\begin{equation}
\intl_0^tM_\eps(s)\suml_{k=1}^n\frac{<(\tilde{A}_{M_l}-\tilde{A})u(s),B_ku(s)>}{|u(s)|_H^2+\eps}dw^k(s)\to
0,l\to\infty,\omega\in\Omega_4,t\in[0,T].\label{eqn:AuxConv-4.6}
\end{equation}
Put $\Omega^\prime=\bigcap\limits_{i=1}^4\Omega_i$. Combining
convergence results \eqref{eqn:AuxConv-1.4},
\eqref{eqn:AuxConv-2.4}, \eqref{eqn:AuxConv-3.4} and
\eqref{eqn:AuxConv-4.6} we prove the Lemma with sequence
$\{M_l\}_{l=1}^{\infty}$ and space $\Omega^\prime$.
\end{proof}

\begin{proof}[Proof of Lemma \ref{lem_est}]
We have by the assumptions \eqref{eqn:CoercivAss},
\eqref{eqn:NonlinearityAss-n}, \eqref{eqn:NonlinearityAss} the
following chain of inequalities
\begin{multline*}
\intl_{0}^t\mathbb{E}_{\mathbb{Q}^\eps}\widetilde{\Lambda}_{\eps}^F(s,u(s))\,ds\leq\intl_{0}^t\mathbb{E}_{\mathbb{Q}^\eps}\widetilde{\Lambda}_{\eps}(u(s))\,ds
+\mathbb{E}_{\mathbb{Q}^\eps}\intl_{0}^t\frac{n(s)|u|\Vert
u\Vert}{|u|^2 +\eps}\,ds
\\
\leq\intl_{0}^t\mathbb{E}_{\mathbb{Q}^\eps}\widetilde{\Lambda}_{\eps}(u(s))\,ds
+(\mathbb{E}_{\mathbb{Q}^\eps}\intl_{0}^tn^2(s)\,ds)^{1/2}(\mathbb{E}_{\mathbb{Q}^\eps}\intl_{0}^t\frac{|u|^2\Vert
u\Vert^2}{(|u|^2+\eps)^2}\,ds)^{1/2}
\\
\leq\intl_{0}^t\mathbb{E}_{\mathbb{Q}^\eps}\widetilde{\Lambda}_{\eps}(u(s))\,ds
+C(\mathbb{E}_{\mathbb{Q}^\eps}\intl_{0}^t\frac{\Vert
u\Vert^2}{|u|^2+\eps}\,ds)^{1/2}
\\
\leq\intl_{0}^t\mathbb{E}_{\mathbb{Q}^\eps}\widetilde{\Lambda}_{\eps}(u(s))\,ds+C(\intl_{0}^t\mathbb{E}_{\mathbb{Q}^\eps}\widetilde{\Lambda}_{\eps}(u(s))-\lambda\,ds)^{1/2}.
\end{multline*}
Therefore, it is enough to estimate from above the term
$\mathbb{E}_{\mathbb{Q}^\eps}\widetilde{\Lambda}_{\eps}(u(s))$.
Because of the assumption \eqref{eqn:B_kWeakBoundAss} we have only
to consider the following function
$\widetilde{\widetilde{\Lambda}}_{\eps}(u(t))$, $t\geq 0$, where
$$\widetilde{\widetilde{\Lambda}}_{\eps}(u)=\frac{\lb (A-\frac{1}{2}\suml_{k=1}^nB_k^*B_k)u,u\rb}{|u|^2+\eps}=\frac{\lb \tilde{A}u,u\rb}{|u|^2+\eps},\;\; u\in V.$$
%and consider only case of selfadjoint $\tilde{A}$.

We will prove that
\begin{equation}
\supl_t\mathbb{E}_{\mathbb{Q}^\eps}\widetilde{\widetilde{\Lambda}}_{\eps}(u(t))<\infty.\label{eqn:WidetildeLambdaEst}
\end{equation}
Since we cannot directly apply the It\^{o} formula to the function
$\widetilde{\widetilde{\Lambda}}_{\eps}$ we will consider finite
dimensional its approximations.

For this aim define %let $\{e_i\}_{i\geq 1}$  be an
%orthonormal basis in $H$ such that $e_i\in D(\hat{A})$ for all $i$. %in which operator $\tilde{A}$ is diagonal
%%(such basis always exists because $\tilde{A}$ is selfadjoint and
%%positive)
%For $N\in\mathbb{N}$ let  $P_N$ be the orthogonal  projection onto
%the space $H_N=\lin\{e_1,\cdots,e_N\}$ and let $Q_N=\id-P_N$.
%Define a linear operator $\tilde{A}_N=P_N^*\tilde{A}P_N$ and a
function $\widetilde{\widetilde{\Lambda^N_\eps}}(u):=\frac{\lb
\tilde{A}_Nu,u\rb}{|u|^2+\eps}$, $u\in H$.  Then, by the Lemma
\ref{lem:AuxRes_1} applied with $C=\tilde{A}_N$,
$\widetilde{\widetilde{\Lambda^N_\eps}}$ is of $C^2$ class and it
has bounded $1^{\rm st}$ and $2^{\rm nd}$ derivatives.
%, compare with the proof of Lemma \ref{lem-psi_ito}, one has the
%following equalities
%\begin{eqnarray*}
%\widetilde{\widetilde{\Lambda^N_\eps}}^\prime(x)h_1&=&\frac{2\lb \tilde{A}_Nx,h_1\rb}{|x|^2+\eps}-\frac{2\lb \tilde{A}_Nx,x\rb\lb x,h_1\rb}{(|x|^2+\eps)^2},\\
%\widetilde{\widetilde{\Lambda^N_\eps}}^{\prime\prime}(x)(h_1,h_2)&=&2\frac{\lb \tilde{A}_Nh_1,h_2\rb}{|x|^2+\eps}-4\frac{\lb \tilde{A}_Nx,h_1\rb\lb x,h_2\rb}{(|x|^2+\eps)^2}\\
%&-&4\frac{\lb \tilde{A}_Nx,h_2\rb\lb x,h_1\rb}{(|x|^2+\eps)^2}-2\frac{\lb \tilde{A}_Nx,x\rb\lb h_2,h_1\rb}{(|x|^2+\eps)^2}\\
%&+&8\frac{\lb \tilde{A}_Nx,x\rb\lb x,h_1\rb\lb
%x,h_2\rb}{(|x|^2+\eps)^3}.
%\end{eqnarray*}
By the It\^o formula and Lemma \ref{lem:AuxRes_1} we have, with
$u=u(t)$ on the RHS,
\begin{eqnarray*}
d\widetilde{\widetilde{\Lambda^N_\eps}}(u(t))&=&\frac{\lb
\tilde{A}_N^\prime u,u\rb}{|u|^2+\eps}\,dt+\frac{2\lb
\tilde{A}_Nu,du\rb}{|u|^2+\eps}-\frac{2\lb \tilde{A}_Nu,u\rb\lb
u,du\rb}{(|u|^2+\eps)^2}\nonumber\\
&+&\suml_{k=1}^n\left(\frac{\lb
\tilde{A}_NB_ku,B_ku\rb}{|u|^2+\eps}-\frac{4\lb
\tilde{A}_Nu,B_ku\rb\lb u,B_ku\rb}{(|u|^2+\eps)^2}\right.\nonumber\\
&-&\left.\frac{\lb
\tilde{A}_Nu,u\rb|B_ku|^2}{(|u|^2+\eps)^2}+\frac{4\lb
\tilde{A}_Nu,u\rb\lb u,B_ku\rb^2}{(|u|^2+\eps)^3}\right)\,dt
\end{eqnarray*}
Because $u$ is a solution of problem \eqref{eqn-1.1} we have,
still with $u=u(t)$,
\begin{eqnarray}
d\widetilde{\widetilde{\Lambda^N_\eps}}(u(t))&=&-\frac{2\lb \tilde{A}_Nu,Au+F(t,u)\rb}{|u|^2+\eps}\,dt+\frac{2\lb \tilde{A}_Nu,u\rb\lb u,Au+F(t,u)\rb}{(|u|^2+\eps)^2}\,dt\nonumber\\
&+&\frac{\lb \tilde{A}_N^\prime
u,u\rb}{|u|^2+\eps}\,dt+\suml_{k=1}^n\frac{\lb
\tilde{A}_NB_ku,B_ku\rb}{|u|^2+\eps}\,dt-\suml_{k=1}^n\frac{\lb
\tilde{A}_Nu,u\rb|B_ku|^2}{(|u|^2+\eps)^2}\,dt
\label{eqn:eqn-1'}\\
&+&\big(\suml_{k=1}^n\frac{2\lb \tilde{A}_Nu,u\rb\lb
u,B_ku\rb}{(|u|^2+\eps)^2}-\suml_{k=1}^n\frac{2\lb
\tilde{A}_Nu,B_ku\rb}{|u|^2+\eps}\big)(\, dw_t^k+2\frac{\lb
u,B_ku\rb}{|u|^2+\eps}\,dt)\,\nonumber
\end{eqnarray}
Therefore,
\begin{eqnarray}
d(M_\eps(t)\widetilde{\widetilde{\Lambda^N_\eps}}(u(t)))&=&M_\eps(t)(-\frac{2\lb \tilde{A}_Nu,Au+F(t,u)\rb}{|u|^2+\eps}\,dt+\frac{2\lb \tilde{A}_Nu,u\rb\lb u,Au+F(t,u)\rb}{(|u|^2+\eps)^2}\,dt\nonumber\\
&+&\frac{\lb \tilde{A}_N^\prime u,u\rb}{|u|^2+\eps}\,dt+\suml_{k=1}^n\frac{\lb \tilde{A}_NB_ku,B_ku\rb}{|u|^2+\eps}\,dt-\suml_{k=1}^n\frac{\lb \tilde{A}_Nu,u\rb|B_ku|^2}{(|u|^2+\eps)^2}\,ds\label{eqn:eqn-1}\\
&-&\suml_{k=1}^n\frac{2\lb \tilde{A}_Nu,B_ku\rb}{|u|^2+\eps}\,
dw_t^k).\nonumber
\end{eqnarray}
The above equality \eqref{eqn:eqn-1} can be rewritten as
\begin{eqnarray}
\label{eqn:eqn-1.2}
d(M_\eps(t)\widetilde{\widetilde{\Lambda^N_\eps}}(u(t)))&=&M_\eps(s)\Big(\frac{\lb
\tilde{A}_Nu,u\rb(2\lb
Au,u\rb-\suml_{k=1}^n|B_ku|^2)}{(|u|^2+\eps)^2} +\frac{\lb
\tilde{A}_N^\prime u,u\rb}{|u|^2+\eps}
\nonumber\\
&-&\frac{2\lb
\tilde{A}_Nu,Au\rb}{|u|^2+\eps}+\suml_{k=1}^n\frac{\lb
\tilde{A}_NB_ku,B_ku\rb}{|u|^2+\eps}\Big)\,ds +
M_\eps(s)\Big(\frac{2\lb \tilde{A}_Nu,u\rb\lb
F(s,u),u\rb}{(|u|^2\eps)^2}\\
&-&\frac{2\lb
\tilde{A}_Nu,F(s,u)\rb}{|u|^2+\eps}\Big)\,ds-M_\eps(s)\suml_{k=1}^n\frac{2\lb
\tilde{A}_Nu,B_ku\rb}{|u|^2+\eps}\, dw^k(s)\nonumber\\
&=&M_\eps(s)\left(\frac{2\lb \tilde{A}_Nu,u\rb\lb
\tilde{A}u,u\rb}{(|u|^2+\eps)^2} -\frac{2\lb
\tilde{A}_Nu,Au\rb}{|u|^2+\eps}\right.\nonumber
\\%\end{eqnarray}\begin{eqnarray}
&+&\frac{2\lb \tilde{A}_Nu,u\rb\lb
F(s,u),u\rb}{(|u|^2+\eps)^2}-\frac{2\lb
\tilde{A}_Nu,F(s,u)\rb}{|u|^2+\eps}+\frac{\lb \tilde{A}_N^\prime
u,u\rb}{|u|^2+\eps}\nonumber
\\%\end{eqnarray}\begin{eqnarray}
&+&\left.\suml_{k=1}^n\frac{\lb
\tilde{A}_NB_ku,B_ku\rb}{|u|^2+\eps}\right)\,ds
-M_\eps(s)\suml_{k=1}^n\frac{2\lb
\tilde{A}_Nu,B_ku\rb}{|u|^2+\eps}\, dw^k(s).
 \nonumber
\end{eqnarray}
The drift term on the right hand side of \eqref{eqn:eqn-1.2} can
be written as:
\begin{eqnarray}
\label{eqn:drift-1}
&&M_\eps\Big(2\widetilde{\widetilde{\Lambda^N_\eps}}\frac{\lb
\tilde{A}u,u\rb}{|u|^2+\eps}-2\frac{\lb
\tilde{A}_Nu,Au\rb}{|u|^2+\eps}+
2\widetilde{\widetilde{\Lambda^N_\eps}}\frac{\lb F(t,u),u\rb}{|u|^2+\eps}-2\frac{\lb \tilde{A}_Nu,F(t,u)\rb}{|u|^2+\eps}+\frac{\lb C_Nu,u\rb}{|u|^2+\eps}\Big)\\
&=&M_\eps\Big(-2\frac{\lb
\tilde{A}_N-\widetilde{\widetilde{\Lambda^N_\eps}}u,\tilde{A}u\rb}{|u|^2+\eps}
-2\frac{\lb
\tilde{A}_N-\widetilde{\widetilde{\Lambda^N_\eps}}u,F(t,u)\rb}{|u|^2+\eps}+\frac{\lb
C_Nu,u\rb}{|u|^2+\eps}\Big), \nonumber
\end{eqnarray}
where
$C_N=\suml_{k=1}^nB_k^*[\tilde{A}_N,B_k]+\tilde{A}_N^\prime(\cdot)$.
The first term on the right hand side \eqref{eqn:drift-1} can be
rewritten as:
\begin{eqnarray}
-2M_\eps\frac{\lb
\tilde{A}_N-\widetilde{\widetilde{\Lambda^N_\eps}}u,\tilde{A}u\rb}{|u|^2+\eps}
&=&M_\eps\Big(-2\frac{|\tilde{A}_N-\widetilde{\widetilde{\Lambda^N_\eps}}u|^2}{|u|^2+\eps}-
2\frac{\lb \tilde{A}_N-\widetilde{\widetilde{\Lambda^N_\eps}}u,(\tilde{A}-\tilde{A}_N)u\rb}{|u|^2+\eps}\nonumber\\
&-&2\frac{\lb
\tilde{A}_N-\widetilde{\widetilde{\Lambda^N_\eps}}u,\widetilde{\widetilde{\Lambda^N_\eps}}u\rb}{|u|^2+\eps}\Big)\nonumber\\
&=&
M_\eps\Big(-2\frac{|\tilde{A}_N-\widetilde{\widetilde{\Lambda^N_\eps}}u|^2}{|u|^2+\eps}-
2\frac{\eps(\widetilde{\widetilde{\Lambda^N_\eps}})^2}{|u|^2+\eps}-
2\frac{\lb
\tilde{A}_N-\widetilde{\widetilde{\Lambda^N_\eps}}u,(\tilde{A}-\tilde{A}_N)u\rb}{|u|^2+\eps}\Big).\nonumber
\end{eqnarray}
Therefore we have
\begin{eqnarray}
&&M_\eps(t)\widetilde{\widetilde{\Lambda^N_\eps}}(u(t))+2\intl_{0}^tM_\eps(s)\frac{|\tilde{A}_N-\widetilde{\widetilde{\Lambda^N_\eps}}u|^2}{|u|^2+\eps}\,ds+
2\eps\intl_{0}^tM_\eps(s)\frac{(\widetilde{\widetilde{\Lambda^N_\eps}})^2}{|u|^2+\eps}\,ds\nonumber\\
&=&M_\eps(0)\widetilde{\widetilde{\Lambda^N_\eps}}(0)-
2\intl_{0}^tM_\eps(s)\frac{\lb
(\tilde{A}_N-\widetilde{\widetilde{\Lambda^N_\eps}})u,(\tilde{A}-\tilde{A}_N)u\rb}{|u|^2+\eps}\,ds-
2\intl_{0}^tM_\eps(s)\frac{\lb (\tilde{A}_N-\widetilde{\widetilde{\Lambda^N_\eps}})u,F(s,u)\rb}{|u|^2+\eps}\,ds\nonumber\\
&+&\intl_{0}^tM_\eps(s)\frac{\lb
C_Nu,u\rb}{|u|^2+\eps}\,ds-\intl_{0}^tM_\eps(s)\suml_{k=1}^n\frac{2\lb
\tilde{A}_Nu,B_ku\rb}{|u|^2+\eps}\, dw^k(s)
=(i)+(ii)+(iii)+(iv)+(v).
 \nonumber
\end{eqnarray}
%If we take mathematical expectation and notice that
%\begin{eqnarray}
%\mathbb{E}|\suml_{k=1}^n\intl_{0}^tM_\eps(s)\frac{2\lb \tilde{A}_Nu,B_ku\rb}{|u|^2+\eps}\, dw^k(s)|^{4/3}\leq\nonumber\\
%C\suml_k\intl_{0}^t\mathbb{E}((M_\eps(s))^{2/3}\frac{|\tilde{A}_Nu|^{2/3}|B_ku|^{2/3}}{(|u|^2+\eps)^{2/3}})\,ds\leq\nonumber\\
%\frac{C}{\eps}(\intl_{0}^t\mathbb{E}M_\eps(s)\,ds)^{2/3}\suml_k(\intl_{0}^t\mathbb{E}|B_ku|^2\, ds)^{1/3}< \infty\nonumber
%\end{eqnarray}
%(here we have used the fact that $A_N$ is bounded and $u\in
%\mathcal{M}^2(0,T;D(\hat{A}))$) we get
%\begin{eqnarray}
%\mathbb{E}_{\mathbb{Q}^\eps}\widetilde{\widetilde{\Lambda^N_\eps}}(u(t))+2\mathbb{E}_{\mathbb{Q}^\eps}\intl_{0}^t\frac{|\tilde{A}_N-\widetilde{\widetilde{\Lambda^N_\eps}}u|^2}{|u|^2+\eps}\,ds+
%2\eps\mathbb{E}_{\mathbb{Q}^\eps}\intl_{0}^t\frac{(\widetilde{\widetilde{\Lambda^N_\eps}})^2}{|u|^2+\eps}=\mathbb{E}_{\mathbb{Q}^\eps}\widetilde{\widetilde{\Lambda^N_\eps}}(0)-\nonumber\\
%2\mathbb{E}_{\mathbb{Q}^\eps}\intl_{0}^t\frac{\lb \tilde{A}_N-\widetilde{\widetilde{\Lambda^N_\eps}}u,(\tilde{A}-\tilde{A}_N)u\rb}{|u|^2+\eps}\,ds-
%2\mathbb{E}_{\mathbb{Q}^\eps}\intl_{0}^t\frac{\lb \tilde{A}_N-\widetilde{\widetilde{\Lambda^N_\eps}}u,F(s,u)\rb}{|u|^2+\eps}\,ds+\nonumber\\
%\mathbb{E}_{\mathbb{Q}^\eps}\intl_{0}^t\frac{\lb C_Nu,u\rb}{|u|^2+\eps}\,ds=(i)+(ii)+(iii)+(iv)\nonumber
%\end{eqnarray}
%\mathbb{E}_{\mathbb{Q}^\eps}
Next  we shall deal with  estimating the term  \textit{(ii)}
above. By the Young inequality we have, with $\eps_1<\frac12$,
\begin{eqnarray*}
\intl_{0}^tM_\eps(s)\frac{\lb
\tilde{A}_N-\widetilde{\widetilde{\Lambda^N_\eps}}u,(\tilde{A}-\tilde{A}_N)u\rb}{|u|^2+\eps}\,ds
&\leq &
\eps_1\intl_{0}^tM_\eps(s)\frac{|\tilde{A}_N-\widetilde{\widetilde{\Lambda^N_\eps}}u|^2}{|u|^2+\eps}\,ds\\
&+&\frac{C}{\eps_1}\intl_{0}^tM_\eps(s)\frac{|(\tilde{A}-\tilde{A}_N)u|^2}{|u|^2+\eps}\,ds.
\end{eqnarray*}
A similar method can be applied to the term \textit{(iii)}. As a
result, we get
\begin{eqnarray}\label{eqn:LambdaEstim-1}
M_\eps(t)\widetilde{\widetilde{\Lambda^N_\eps}}(u(t))&+&\intl_{0}^tM_\eps(s)\frac{|\tilde{A}_N-\widetilde{\widetilde{\Lambda^N_\eps}}u|^2}{|u|^2+\eps}\,ds
+2\eps\intl_{0}^tM_\eps(s)\frac{(\widetilde{\widetilde{\Lambda^N_\eps}})^2}{|u|^2+\eps}
\nonumber\\
&\leq&  M_\eps(0)\widetilde{\widetilde{\Lambda^N_\eps}}(0)
+\frac{C}{\eps_1}\intl_{0}^tM_\eps(s)\frac{|(\tilde{A}-\tilde{A}_N)u|^2}{|u|^2+\eps}\,ds
+\frac{C}{\eps_2}\intl_{0}^tM_\eps(s)\frac{|F(s,u)|^2}{|u|^2+\eps}\,ds\\
&+&\intl_{0}^tM_\eps(s)\frac{\lb
C_Nu,u\rb}{|u|^2+\eps}\,ds-\intl_{0}^tM_\eps(s)\suml_{k=1}^n\frac{2\lb
\tilde{A}_Nu,B_ku\rb}{|u|^2+\eps}\, dw^k(s).
\nonumber\end{eqnarray} Let us estimate the term \textit{(iv)}. By
the assumption \textbf{AC4} and the definition of operator $C_N$
we have the following chain of inequalities:
\begin{multline*}
|\lb C_Nx,x\rb|\leq |\lb
Cx,x\rb|+|\tilde{A}_N^\prime(\cdot)|_{\mathcal{L}(V,V^\prime)}\Vert
x\Vert^2+|\suml_k\lb (\tilde{A}_N-\tilde{A})B_kx,B_kx\rb|\\
+ |\lb (\tilde{A}_N-\tilde{A})x,\suml_kB_k^*B_kx\rb| \leq
K_1(\cdot)|x|^2+(K_2(\cdot)+|\tilde{A}^\prime(\cdot)|_{\mathcal{L}(V,V^\prime)})|\lb
\tilde{A}x,x\rb|\\+|\suml_k\lb
\tilde{A}Q_NB_kx,Q_NB_kx\rb|+|(\tilde{A}_N-\tilde{A})x\Vert
x|_{D(\hat{A})} \leq K_1(\cdot)|x|^2\\+
(K_2(\cdot)+|\tilde{A}^\prime(\cdot)|_{\mathcal{L}(V,V^\prime)})|\lb
\tilde{A}x,x\rb|
%\end{eqnarray*}
%\begin{eqnarray*}
+|(\tilde{A}_N-\tilde{A})x\Vert x|_{D(\hat{A})}+\suml_k\Vert
Q_NB_kx\Vert^2.
\end{multline*}
where assumption \eqref{eqn:CWeakBoundAss} has been used in second
inequality and assumption \eqref{eqn:AStrongBoundAss} in last
inequality. Therefore
\begin{multline}\label{eqn:LambdaEstim-2}
\intl_{0}^tM_\eps(s)\frac{\lb C_Nu,u\rb}{|u|^2+\eps}\,ds\leq
\intl_{0}^tK_1(s)M_\eps(s)\,ds+
\intl_{0}^t(K_2(s)+|\tilde{A}^\prime(s)|_{\mathcal{L}(V,V^\prime)})M_\eps(s)\widetilde{\widetilde{\Lambda^N_\eps}}(u(s))\,ds\\
+\Big(\intl_{0}^tM_\eps(s)\frac{|(\tilde{A}_N-\tilde{A})u|^2}{|u|^2+\eps}\,ds\Big)^{1/2}
\Big(\intl_{0}^tM_\eps(s)\frac{|\tilde{A}u|^2}{|u|^2+\eps}\,ds\Big)^{1/2}+
\intl_{0}^tM_\eps(s)\frac{\suml_k\Vert
Q_NB_ku\Vert^2}{|u|^2+\eps}\,ds.
\end{multline}
We will denote
\begin{eqnarray*}
K_3(\eps,N)&=&\intl_{0}^tM_\eps(s)\frac{|(\tilde{A}_N-\tilde{A})u|^2}{|u|^2+\eps}\,ds,\\
K_4(\eps,N)&=&\intl_{0}^tM_\eps(s)\frac{\suml_k\Vert
Q_NB_ku\Vert^2}{|u|^2+\eps}\,ds,\\
K_5(\eps)&=&\intl_{0}^tM_\eps(s)\frac{|\tilde{A}u|^2}{|u|^2+\eps}\,ds,\\
K_6(s)&=&|\tilde{A}^\prime(s)|_{\mathcal{L}(V,V^\prime)}.
\end{eqnarray*}
 Combining
\eqref{eqn:LambdaEstim-1}, \eqref{eqn:LambdaEstim-2} with the
assumption \eqref{eqn:NonlinearityAss} we get
\begin{multline}\label{eqn:ineq-1}
M_\eps(t)\widetilde{\widetilde{\Lambda^N_\eps}}(u(t))+\intl_{0}^tM_\eps(s)\frac{|\tilde{A}_N-\widetilde{\widetilde{\Lambda^N_\eps}}u|^2}{|u|^2+\eps}\,ds+
2\eps\intl_{0}^tM_\eps(s)\frac{(\widetilde{\widetilde{\Lambda^N_\eps}})^2}{|u|^2+\eps}\\
\leq
M_\eps(0)\widetilde{\widetilde{\Lambda^N_\eps}}(0)+\intl_{0}^tK_1(s)M_\eps(s)\,ds+
\frac{C}{\eps_1}K_3(\eps,N)+(K_3(\eps,N)K_5(\eps))^{1/2}+K_4(\eps,N)\\
+\frac{C}{\eps_2}\intl_{0}^t(n^2(s)+K_2(s)+K_6(s))M_\eps(s)\widetilde{\widetilde{\Lambda^N_\eps}}(u(s))\,ds-\intl_{0}^tM_\eps(s)\suml_{k=1}^n\frac{2\lb
\tilde{A}_Nu,B_ku\rb}{|u|^2+\eps}\, dw^k(s)
\end{multline}
%In the proof of inequality \eqref{eqn:ineq-1} we have not used the
%assumption \eqref{eqn:SolutionAss-1'}. It will be later used in
%the proof of Theorem \ref{thm:ExistenceSpectrLimit}.

%Thus, from the Gronwall lemma follows that
%\begin{eqnarray}
%\mathbb{E}_{\mathbb{Q}^\eps}\widetilde{\widetilde{\Lambda^N_\eps}}(u(t))\leq(\mathbb{E}_{\mathbb{Q}^\eps}\widetilde{\widetilde{\Lambda^N_\eps}}(0)+K_1t+\frac{C}{\eps_1}K_3(\eps,N)+(K_3K_5(\eps))^{1/2}+K_4(\eps,N))\times\nonumber\\
%e^{C_1T+C_2\intl_0^Tn^2(s)\,ds}.\label{eqn:epsbound-1}
%\end{eqnarray}
%Now we notice that the expression
%$\frac{|(\tilde{A}-\tilde{A}_N)u|^2}{|u|^2+\eps}$ can be bounded
%from above by $\frac{|\tilde{A}u|^2}{\eps}$. Then we have
%\begin{eqnarray*}
%\mathbb{E}K_3(\eps,N)&=&\mathbb{E}\intl_{0}^tM_\eps(s)\frac{|(\tilde{A}-\tilde{A}_N)u|^2}{|u|^2+\eps}\,ds
%\leq
%\frac{C}{\eps}(\mathbb{E}((M_\eps(t))^{1+\frac{2}{\delta_0}}))^{\frac{\delta_0}{2+\delta_0}}
%(\mathbb{E}\intl_{0}^t|\tilde{A}u|^{2+\delta_0}\,ds)^{\frac{2}{2+\delta_0}}<
%\infty,
%\end{eqnarray*}
%where last inequality is consequence of  the assumption
%\eqref{eqn:SolutionAss-1'}. Similarly, we have
%$\mathbb{E}K_4(\eps,N)< \infty$ uniformly w.r.t. $N$ and
%$\mathbb{E}K_5(\eps)< \infty$. Therefore,
By assumption \eqref{eqn:SolutionAss-1'} we can find
$\Omega_5\subset\Omega$, $\mathbb{P}(\Omega_5)=1$ such that
\begin{equation}
K_5(\eps)<\infty,u\in L^2(0,T;D(\hat{A})), \omega\in
\Omega_5.\label{eqn:AuxConv-5.1}
\end{equation}
By Lemma \ref{lem:AuxRes_2} we can find
$\Omega^\prime\subset\Omega$, $\mathbb{P}(\Omega^\prime)=1$ and
sequence $\{M_l\}_{l=1}^{\infty}$ such that $\forall
\omega\in\Omega^\prime$, $t\in[0,T]$ %$K_3(\eps,N)< \infty$, $K_4(\eps,N)<
%\infty$. Moreover, we notice that $K_3(\eps,N)$,
%$K_4(\eps,N)$, $\widetilde{\widetilde{\Lambda^N_\eps}}$ are
%monotone w.r.t. $N$.
%Tending $N$ to $\infty$ in \eqref{eqn:ineq-1} and noticing that by
%the Lebesgue Dominated Convergence Theorem
\begin{eqnarray}
&&\liml_{l\to\infty}K_3(\eps,M_l)=0,
\liml_{l\to\infty}K_4(\eps,M_l)=0,\\
&&\liml_{l\to\infty}\widetilde{\widetilde{\Lambda^{M_l}_\eps}}(u(t))=\widetilde{\widetilde{\Lambda}}_{\eps}(t),\nonumber\\
&&\liml_{l\to\infty}\intl_0^tM_\eps(s)\suml_{k=1}^n\frac{<\tilde{A}_{M_l}u(s),B_ku(s)>}{|u(s)|_H^2+\eps}dw^k(s)
=\intl_0^tM_\eps(s)\suml_{k=1}^n\frac{<\tilde{A}u(s),B_ku(s)>}{|u(s)|_H^2+\eps}dw^k(s).\nonumber
\end{eqnarray}
Therefore, applying \eqref{eqn:ineq-1} with subsequence
$\{M_l\}_{l=1}^{\infty}$ and tending $l\to \infty$ in equality
\eqref{eqn:ineq-1} we infer that
%for $\omega\in\Omega_5\cap\Omega^\prime$ that
\begin{eqnarray}\label{eqn:ineq-2}
M_\eps(t)\widetilde{\widetilde{\Lambda}}_{\eps}(t)&+&\intl_{0}^tM_\eps(s)\frac{|(\tilde{A}-\widetilde{\widetilde{\Lambda}}_{\eps})u|^2}{|u|^2+\eps}\,ds+
2\eps\intl_{0}^tM_\eps(s)\frac{(\widetilde{\widetilde{\Lambda}}_{\eps})^2}{|u|^2+\eps}\,ds
\nn\\
&\leq&
M_\eps(0)\widetilde{\widetilde{\Lambda}}_{\eps}(0)+\intl_{0}^tK_1(s)M_\eps(s)\,ds
+\frac{C}{\eps_2}\intl_{0}^t(n^2(s)+K_2(s)\\
&+&
K_6(s))M_\eps(s)\widetilde{\widetilde{\Lambda}}_{\eps}(u(s))\,ds
-\intl_{0}^tM_\eps(s)\suml_{k=1}^n\frac{2\lb
\tilde{A}u,B_ku\rb}{|u|^2+\eps}\, dw^k(s)
\nn,\omega\in\Omega_5\cap\Omega^\prime .
\end{eqnarray}
Let $X_\eps=(X_{\eps}(t))_{t\geq 0}$--solution of the following
SDE
\begin{eqnarray}\label{eqn:compl-1}
X_{\eps}(t)&=&M_\eps(0)\widetilde{\widetilde{\Lambda}}_{\eps}(0)+\intl_{0}^tK_1(s)M_\eps(s)\,ds+
\frac{C}{\eps_2}\intl_{0}^t(n^2(s)+K_2(s)+K_6(s))X_{\eps}(s)\,ds\nonumber\\
&-&\suml_{k=1}^n\intl_{0}^tM_\eps(s)\frac{2\lb
\tilde{A}u,B_ku\rb}{|u|^2+\eps}\, dw^k(s)
\end{eqnarray}
Then, we have that a.a.%\Red{In which sense?}
\begin{equation}
M_\eps(t)\widetilde{\widetilde{\Lambda}}_{\eps}(t)+\intl_{0}^tM_\eps(s)\frac{|(\tilde{A}-\widetilde{\widetilde{\Lambda}}_{\eps})u|^2}{|u|^2+\eps}\,ds\leq
X_{\eps}(t).\label{eqn:LambdaEstim-3}
\end{equation}
Indeed, it is enough to subtract from the inequality
\eqref{eqn:ineq-2} the identity \eqref{eqn:compl-1} and then use
the Gronwall Lemma. On the other hand equation \eqref{eqn:compl-1}
can be solved explicitly. It's unique  solution is given the
following formula explicitly and we have
\begin{eqnarray}
X_{\eps}(t)&=&M_\eps(0)\widetilde{\widetilde{\Lambda}}_{\eps}(0)e^{\intl_{0}^t(n^2(s)+K_2(s)+K_6(s))\,ds}\\
\nonumber
&+&\intl_{0}^te^{\intl_s^t(n^2(\tau)+K_2(\tau)+K_6(\tau))d\tau}\, K_1(s)M_\eps(s)\,ds\label{eqn:X_eps}\\
\nonumber &-&\suml_{k=1}^n
\intl_{0}^te^{\intl_s^t(n^2(\tau)+K_2(\tau)+K_6(\tau))d\tau}
M_\eps(s)\frac{2\lb \tilde{A}u,B_ku\rb}{|u|^2+\eps}\, dw^k(s).
\end{eqnarray}
Denote
\begin{equation}\label{eqn_def_L^eps}
L_{\eps}(t)=\suml_{k=1}^n\intl_{0}^te^{\intl_s^t(n^2(\tau)+K_2(\tau)+K_6(\tau))d\tau}M_\eps(s)\frac{2\lb
\tilde{A}u,B_ku\rb}{|u|^2+\eps}\, dw^k(s), \; t\geq 0.
\end{equation}
Let us show that this definition is correct. It is enough to show
that a.a.
\begin{equation*}
\suml_{k=1}^n\intl_{0}^te^{2\intl_s^t(n^2(\tau)+K_2(\tau)+K_6(\tau))d\tau}|M_\eps(s)|^2\frac{2\lb
\tilde{A}u,B_ku\rb^2}{(|u|^2+\eps)^2}\,ds<\infty.
\end{equation*}
We have
\begin{eqnarray}
\suml_{k=1}^n\intl_{0}^te^{2\intl_s^t(n^2(\tau)+K_2(\tau)+K_6(\tau))d\tau}(M_\eps(s))^2\frac{2\lb
\tilde{A}u,B_ku\rb^2}{(|u|^2+\eps)^2}\,ds \leq
e^{2\intl_{0}^t(n^2(\tau)+K_2(\tau)+K_6(\tau))d\tau}\nonumber\\
\times\supl_{s\leq t}|M_\eps(s)|^2\frac{1}{\eps^2}\supl_{s\leq
t}\suml_{k=1}^n|B_k u(s)|^2\intl_{0}^t|\tilde{A}u(s)|^2\,ds \leq
Ce^{2\intl_{0}^t(n^2(\tau)+K_2(\tau)+K_6(\tau))d\tau}\nonumber\\
\times\supl_{s\leq t}|M_\eps(s)|^2\frac{1}{\eps^2}\supl_{s\leq
t}(\Vert u(s)\Vert^2+|u(s)|^2)\intl_{0}^t|\tilde{A}u(s)|^2\,ds,
\end{eqnarray}
and the result follows from assumptions \eqref{eqn:SolutionAss-1'}
and \eqref{eqn:SolutionAss-2}.

Let us notice that formula \eqref{eqn:X_eps} can be rewritten as
follows
\begin{eqnarray}\label{eqn:X_1-1}
X_{\eps}(t)&=&M_\eps(0)\widetilde{\widetilde{\Lambda}}_{\eps}(0)e^{\intl_{0}^t(n^2(s)+K_2(s)+K_6(s))\,ds}\\
&+&\intl_{0}^te^{\intl_s^t(n^2(\tau)+K_2(\tau)+K_6(\tau))d\tau}K_1(s)M_\eps(s)\,ds-L_{\eps}(t),\,t\geq
0. \nonumber
\end{eqnarray}
By the definition \eqref{eqn_def_L^eps}, the process $L^{\eps}$ is
a local martingale. We will show that in our assumptions it is
martingale. By proposition \ref{prop:MartingaleCrit} it is enough
to show that there exists $\delta>0$, $K<\infty$
$$
\mathbb{E}|L_{\tau}^{\eps}|^{1+\delta}\leq K,\,\, \tau\in[0,T].
$$
Let $p_i$, $i=1,2,3,4$ be real numbers such that
$\suml_{i=1}^4\frac{1}{p_i}=1$, $p_i>1$, $i=1,2,3,4$. By
Burkholder and H{\"o}lder inequalities we infer that for $t\geq
0$,
\begin{eqnarray*}
\mathbb{E}|L_{\eps}(t)|^{1+\delta} &\leq &
C\mathbb{E}\left(\intl_{0}^te^{2\intl_s^t(n^2+K_2+K_6)d\tau}
(M_\eps(s))^2\suml_{k=1}^n\frac{|\tilde{A}u|^2|B_ku|^2}{(|u|^2+\eps)^2}\,ds\right)^{(1+\delta)/2}
\\
&\leq&
\frac{C}{\eps^{1+\delta}}\mathbb{E}\Big[e^{(1+\delta)\intl_{0}^t(n^2+K_2+K_6)d\tau}\supl_{s\in
[0,t]}(M_\eps(s))^{(1+\delta)}
\\
&\times & \supl_{s\in
[0,t]}\suml_{k=1}^n|B_ku(s)|^{(1+\delta)}(\intl_0^t|\tilde{A}u|^2\,ds)^{(1+\delta)/2}\Big]
\\
&\leq&
C(\eps)\left(\mathbb{E}e^{(1+\delta)p_1\intl_{0}^t(n^2+K_2+K_6)d\tau}\right)^{\frac{1}{p_1}}
(\mathbb{E}\supl_{s\in
[0,t]}(M_\eps(s))^{(1+\delta)p_2})^{\frac{1}{p_2}}
\\
&\times & (\mathbb{E}\supl_{s\in
[0,t]}\suml_{k=1}^n|B_ku(s)|^{(1+\delta)p_3})^{\frac{1}{p_3}}
(\mathbb{E}(\intl_{0}^t|\tilde{A}u|^2\,ds)^{(1+\delta)p_4/2})^{\frac{1}{p_4}}=:S.
\end{eqnarray*}
Choose $\delta\in
(0,\frac{\delta_0\kappa-(4+2\delta_0)}{4+2\delta_0+\kappa(4+\delta_0)})$.
With a special choice of exponents $p_i,i=1,\ldots,4$:
\begin{equation*}
p_1=\frac{\kappa}{1+\delta}, p_3=\frac{2+\delta_0}{1+\delta},
p_4=\frac{2}{1+\delta},
p_2=1/\left(1-(\frac{1}{p_1}+\frac{1}{p_3}+\frac{1}{p_4})\right),
\end{equation*}
we obtain
\begin{eqnarray}
S&=&C(\eps)\left(\mathbb{E}e^{\kappa\intl_{0}^t(n^2+K_2+K_6)d\tau}\right)^{\frac{1}{p_1}}
\left(\mathbb{E}\supl_{s\in
[0,t]}M_\eps(s)^{(1+\delta)p_2}\right)^{\frac{1}{p_2}}\nonumber
\\
&\times & \left(\mathbb{E}\supl_{s\in
[0,t]}\suml_{k=1}^n|B_ku(s)|^{2+\delta_0}\right)^{\frac{1}{p_3}}
\left(\mathbb{E}\intl_{0}^t|\tilde{A}u|^2\,ds\right)^{\frac{1}{p_4}}.\label{eqn:AuxEquality-1.1}
\end{eqnarray}
Notice that by Assumption \eqref{eqn:NonlinearityAss-n} the first
factor in the product \eqref{eqn:AuxEquality-1.1} is finite.
Furthermore, by Doob inequality and Assumption
\eqref{eqn:B_kWeakBoundAss} the second factor is finite. The third
factor is finite by Assumptions \eqref{eqn:CoercivAss},
\eqref{eqn:AStrongBoundAss} and \eqref{eqn:SolutionAss-2}. The
last factor in the product \eqref{eqn:AuxEquality-1.1} is finite
by Assumption \eqref{eqn:SolutionAss-1'}. Thus, $S<\infty$ and
hence, the process $L^{\eps}$ is martingale. In particular,
$\mathbb{E}L_{\eps}(t)=0$ for all $t\geq 0$.

It follows from \eqref{eqn:X_1-1}, \eqref{eqn:LambdaEstim-3} and
the
H{\"o}lder %\Red{Please note how to type the name H{\"o}lder. }
inequality that
\begin{eqnarray}
\mathbb{E}M_\eps(t)\widetilde{\widetilde{\Lambda}}_{\eps}(t) &\leq
&\mathbb{E}X_{\eps}(t)\leq C(
\mathbb{E}\widetilde{\widetilde{\Lambda}}_{\eps}(0)^{1+\delta_0}+t\Vert
K_1\Vert_{L^2(0,T)})\mathbb{E}e^{\kappa(\delta_0)\intl_{0}^tn^2(s)\,ds}\nonumber\\
&\leq & C\left(\frac{\mathbb{E}\Vert
u(0)\Vert^{1+\delta_0}}{c^{1+\delta_0}}+t\Vert
K_1\Vert_{L^2(0,T)}\right)\mathbb{E}e^{\kappa(\delta_0)\intl_{0}^tn^2(s)\,ds}.
\label{eqn:epsbound-2}
\end{eqnarray}
%From Fatou lemma follows that
%$\mathbb{E}_{\mathbb{Q}}\widetilde{\Lambda}(t)\leq\supl_{\eps>
%0}\mathbb{E}_{\mathbb{Q}^\eps}\widetilde{\widetilde{\Lambda}}_{\eps}(t)$.
Therefore, we get the estimate \eqref{eqn:WidetildeLambdaEst} from
\eqref{eqn:epsbound-2} and \eqref{eqn:SolutionAss-2}. Hence the
proof of Lemma \ref{lem_est} is complete.
\end{proof}
\begin{proof}[Proof of Theorem
\ref{thm:BackwardUniqueness}] As mentioned earlier completion of
the proof of Lemma \ref{lem_est} also completes the proof of the
Theorem.
\end{proof}
\begin{proof}[Proof of Theorem \ref{thm:NonlinearBackwardUniqueness}]

We will argue by contradiction. Suppose that the assertion of the
Theorem is not true. Then, because the process $u$ is adapted, we
will be able to find   $t_0\in[0,T)$, an event
$R\in\mathcal{F}_{t_0} $ and  a constant $c>0$ such that
$\mathbb{P}(R)>0$ and
\begin{equation}
|u(t_0,\omega)|\geq c>0,\; \omega\in R.\label{eqn:NLowerBoundSol}
\end{equation}

Without loss of generality we can  assume  that $\mathbb{P}(R)=1$
and $t_0=0$. Otherwise, we can consider instead of measure
$\mathbb{P}$ the conditional measure
$\mathbb{P}_R:=\frac{\mathbb{P}(\cdot\cap R)}{\mathbb{P}(R)}$.

Suppose  that there exists a constant $c>0$ such that
$$|u(t)|^2\geq c, \; \mbox{for all } t\in [0,T].$$
Then, by taking $t=T$, we infer that $|u(T)|^2>0$ what is a clear
contradiction with the assumption that  $u(T)=0$,
$\mathbb{P}$-a.s..

Now we shall  prove that such a constant exists.

Let $\phi_{r}:\mathbb{R}\to \mathbb{R}$, $r>0$ a mollifying
function such that $\phi_r\in C_b^{\infty}(\mathbb{R})$ and
\begin{equation*}
    \phi_r(x)=\left\{\begin{array}{rcl}
    && 1,\mbox{ if }|x|\geq r\\
    &&0\leq \phi_r(x)\leq 1,\mbox{ if }r/2< |x|< r\\
    &&0,\mbox{ if }|x|\geq r/2
    \end{array}\right.
\end{equation*}

Next let us fix $r>0$ and  define a process $\psi^{r}$ by
\begin{equation}
\psi^{r}(t)=-\frac{1}{2}\phi_{r}(u(t))\log{|u(t)|^2},\;
t\in[0,T].\label{eqn:Psi_Def_2}
\end{equation}
Now for any $r\geq 0$ we define stopping time $\tau_{r}$ as
follows
\begin{equation}
\tau_{r}(\omega)=\inf\{t\in[0,T]:|u(t,\omega)|\leq
r\},\;\omega\in\Omega.\label{eqn:Tau_Def}
\end{equation}
Note that $\tau_{r}$ is
well defined since $u(T)=0$ a.s. and $\tau_{r}\leq
\tau_{\delta}\leq \tau_0\leq T$ if $0\leq\delta\leq r$.

As in the Theorem \ref{thm:BackwardUniqueness} we have the
following result.
\begin{lemma}\label{lem-psi_ito_2}
For every $r>0$ the process  $\psi^{r}$ defined in
\eqref{eqn:Psi_Def_2} is an It\^o process and
\begin{eqnarray}
\psi^{r}(t\wedge\tau_{r})&=&\psi_{r}(0)+\intl_0^{t\wedge\tau_{r}}\suml_{k=1}^n\frac{\lb u,B_ku\rb }{|u|^2}\, dw^k(s) \nonumber\\
+\, \intl_0^{t\wedge\tau_{r}}\Big(\suml_{k=1}^n \frac{\lb
u,B_ku\rb^2}{|u|^4} &+& \frac{\lb
(A-\frac{1}{2}\suml_{k=1}^nB_k^*B_k)u+F(s,u),u\rb}{|u|^4}\Big)\,ds,\;t\geq
0.\label{eqn:NAux-2}
\end{eqnarray}
\end{lemma}
Combining equality \eqref{eqn:Aux-2} and definition
\ref{eqn:LA_eps_F_def} we infer that
\begin{equation}\label{eqn:NDerivPsi}
\psi^{r}(t\wedge
\tau_{r})=\psi_{r}(0)+\intl_0^{t\wedge\tau_{r}}\widetilde{\Lambda}^F(s,u(s))\,ds+\intl_0^{t\wedge\tau_{r}}\suml_{k=1}^n\frac{\lb
u,B_ku\rb}{|u|^2}\, dw^k(s)
\end{equation}

Suppose for the time being, that the following result is true.
\begin{lemma}\label{lem_est_N}
In the above framework we have, $\mathbb{P}$-a.s.
$$\intl_0^{\tau_0}\widetilde{\Lambda}^F(s,u(s))ds<\infty.$$
\end{lemma}
Then it follows from Assumption \textbf{(AC3)} that
$\intl_0^T\suml_{k=1}^n\frac{\lb u,B_ku\rb}{|u|^2}\, dw^k(s)$
exists a.e. and
\begin{equation}\label{eqn:AuxIneq_1}
    \mathbb{E}|\intl_0^T\suml_{k=1}^n\frac{\lb
u,B_ku\rb}{|u|^2+\eps}\,
dw^k(s)|^2\leq\intl_0^T\mathbb{E}\suml_{k=1}^n\frac{|\lb
u,B_ku\rb|^2}{|u|^4}\, ds<\infty.
\end{equation}
Therefore we infer that
\begin{equation*}
\log{|u(t\wedge \tau_{r})|^2}\geq \log{|u(0)|^2}-K,\; t\in [0,T],
\end{equation*}
where $K=\intl_{0}^{\tau_0} \widetilde{\Lambda}^F(s,u(s))\,
ds+\intl_0^T\suml_{k=1}^n\frac{\lb u,B_ku\rb}{|u|^2}\, dw^k(s)$.

Hence
\begin{equation}
|u(t\wedge \tau_{r})|^2 = e^{\log{(|u(t\wedge \tau_{r})|^2)}} \geq
e^{\log{|u(0)|^2}-K}=e^{-K}|u(0)|^2\del{>0}.\label{eqn:NAuxEstimate-1}
\end{equation}
Tend $r$ to $0$. Therefore,
\begin{equation}
|u(\tau_0)|^2\geq
\frac{1}{2}e^{-K}|u(0)|^2>0,\,\label{eqn:NAuxEstimate-3}
\end{equation}
what contradicts our assumption that $u(\tau_0)=0$ and the Theorem
follows. \end{proof}

Hence, it only remains to prove Lemma \ref{lem_est_N}.

\begin{proof}[Proof of Lemma \ref{lem_est_N}]
We have by the assumptions \textbf{(AC2)},
\eqref{eqn:NNonlinearityAss-n}, \eqref{eqn:NNonlinearityAss} the
following chain of inequalities
\begin{multline*}
\intl_{0}^t\widetilde{\Lambda}^F(s,u(s))\,ds\leq\intl_{0}^t\widetilde{\Lambda}(u(s))\,ds
+\intl_{0}^t\frac{n(s)|u|\Vert u\Vert}{|u|^2}\,ds
\\
\leq\intl_{0}^t\widetilde{\Lambda}(u(s))\,ds
+(\intl_{0}^tn^2(s)\,ds)^{1/2}(\intl_{0}^t\frac{|u|^2\Vert
u\Vert^2}{|u|^4}\,ds)^{1/2}
\\
\leq\intl_{0}^t\widetilde{\Lambda}(u(s))\,ds
+C(\intl_{0}^t\frac{\Vert u\Vert^2}{|u|^2}\,ds)^{1/2}
\\
\leq\intl_{0}^t\widetilde{\Lambda}(u(s))\,ds+C(\intl_{0}^t\widetilde{\Lambda}(u(s))-\lambda\,ds)^{1/2}.
\end{multline*}
Therefore, it is enough to estimate from above the term
$\intl_0^{\tau_0}\widetilde{\Lambda}(s,u(s))ds$. Because of the
assumption \eqref{eqn:B_kWeakBoundAss} we have only to consider
the following function
$\intl_0^{\tau_0}\widetilde{\widetilde{\Lambda}}(u(s))ds$, where
$$\widetilde{\widetilde{\Lambda}}(u)=\frac{\lb (A-\frac{1}{2}\suml_{k=1}^nB_k^*B_k)u,u\rb}{|u|^2}=\frac{\lb \tilde{A}u,u\rb}{|u|^2},\;\; u\in V.$$
%and consider only case of selfadjoint $\tilde{A}$.

We will prove that
\begin{equation}
\supl_{t\in[0,\tau_0)}\widetilde{\widetilde{\Lambda}}(u(t))<\infty\mbox{
a.s.}.\label{eqn:NWidetildeLambdaEst}
\end{equation}

Fix $\eps>0$. By the same argument as in the proof of the Theorem
\ref{thm:BackwardUniqueness} we get inequality
\begin{eqnarray}\label{eqn:Nineq-3}
M_\eps(t)\widetilde{\widetilde{\Lambda}}_{\eps}(t)&+&\intl_{\tau}^tM_\eps(s)\frac{|(\tilde{A}-\widetilde{\widetilde{\Lambda}}_{\eps})u|^2}{|u|^2+\eps}\,ds+
2\eps\intl_{\tau}^tM_\eps(s)\frac{(\widetilde{\widetilde{\Lambda}}_{\eps})^2}{|u|^2+\eps}\,ds
\nn\\
&\leq&
M_\eps(0)\widetilde{\widetilde{\Lambda}}_{\eps}(0)+\intl_{\tau}^tK_1(s)M_\eps(s)\,ds
+\frac{C}{\eps_2}\intl_{\tau}^t(n^2(s)+K_2(s)\\
&+& K_6(s))M_\eps(s)\widetilde{\widetilde{\Lambda}}_{\eps}(s)\,ds
-\intl_{\tau}^tM_\eps(s)\suml_{k=1}^n\frac{2\lb
\tilde{A}u,B_ku\rb}{|u|^2+\eps}\, dw^k(s), t\geq\tau.
\nn%,\omega\in\Omega^\prime, \mathbb{P}(\Omega^{\prime})=1.
\end{eqnarray}
Notice that $K_1(s)=0$, $s\in[0,T]$ by assumptions of the Theorem.
Consequently, for $t\geq\tau$,
\begin{equation}
\label{eqn:NX_tEstim-1}
M_\eps(t)\widetilde{\widetilde{\Lambda}}_{\eps}(t)+\intl_{\tau}^tM_\eps(s)\frac{|(\tilde{A}-\widetilde{\widetilde{\Lambda}}_{\eps})u|^2}{|u|^2+\eps}\,ds\leq
(M_\eps(\tau)\widetilde{\widetilde{\Lambda}}_{\eps}(\tau))e^{\intl_{\tau}^t(n^2(s)+K_2(s)+K_6(s))\,ds}-L_{\tau}^{\eps}(t).
\end{equation}
where,
\begin{equation}
L_{\tau}^{\eps}(t)=\intl_{\tau}^te^{\intl_s^t(n^2(r)+K_2(r)+K_6(r))dr}M_\eps(s)\suml_{k=1}^n\frac{2\lb
\tilde{A}u,B_ku\rb}{|u|^2+\eps}\, dw^k(s),t\geq\tau.
\end{equation}
Let us  denote, for $t\geq 0$,
\begin{eqnarray*}
S_{\eps}(t)&=&e^{-\intl_0^t(n^2(r)+K_2(r)+K_6(r))dr}M_\eps(t)\widetilde{\widetilde{\Lambda}}_{\eps}(t),\\
N_{\eps}(t)&=&e^{-\intl_0^t(n^2(r)+K_2(r)+K_6(r))dr}\intl_{\tau}^tM_\eps(s)\frac{|(\tilde{A}-\widetilde{\Lambda}_{\eps})u|^2}{|u|^2+\eps}\,ds.
\end{eqnarray*}
Multiplying inequality \eqref{eqn:NX_tEstim-1} by
$e^{-\intl_0^t(n^2(r)+K_2(r)+K_6(r))dr}$  we infer that
\begin{equation}
\label{eqn:NS_t-1} S_{\eps}(t)+N_{\eps}(t)\leq
S_{\eps}(\tau)-2\intl_{\tau}^te^{-\intl_0^s(n^2(r)+K_2(r)+K_6(r))dr}M_\eps(s)\suml_{k=1}^n\frac{2\lb
\tilde{A}u,B_ku\rb}{|u|^2+\eps}\, dw^k(s).
\end{equation}
Therefore, by the definition of $S_{\eps}$, we infer that
\begin{equation}\label{eqn:NS_t-1'}
S_{\eps}(t)+N_{\eps}(t)\leq
S_{\eps}(\tau)-2\intl_{\tau}^tS_{\eps}(s)\suml_{k=1}^n\frac{\lb
\tilde{A}u,B_ku\rb}{|\lb \tilde{A}u,u\rb|}\, dw^k(s).
\end{equation}
Since $N_{\eps}\geq 0$, by the Comparison Theorem for the one
dimensional diffusions, see e.g. Theorem 1.1 p. 352 in
\cite{[Ikeda-1981]}, we have for $t\geq \tau$, $\mathbb{P}$-a.s.
\begin{equation}\label{eqn:NS_t-1''}
S_{\eps}(t)\leq
S_{\eps}(\tau)e^{-2\intl_{\tau}^t\suml_{k=1}^n\frac{\lb
\tilde{A}u,B_ku\rb}{|\lb \tilde{A}u,u\rb|}\, dw^k(s)
-2\intl_{\tau}^t\suml_{k=1}^n\frac{|\lb
\tilde{A}u,B_ku\rb|^2}{|\lb \tilde{A}u,u\rb|^2}\,ds}.
\end{equation}
Hence
\begin{equation}\label{eqn:NS_t-11''}
\widetilde{\widetilde{\Lambda}}_{\eps}(t)\leq
\widetilde{\widetilde{\Lambda}}_{\eps}(\tau)\frac{M_\eps(\tau)}{M_\eps(t)}e^{\intl_{\tau}^t(n^2(r)+K_2(r)+K_6(r))dr}
e^{-2\intl_{\tau}^t\suml_{k=1}^n\frac{\lb \tilde{A}u,B_ku\rb}{|\lb
\tilde{A}u,u\rb|}\, dw^k(s)
-2\intl_{\tau}^t\suml_{k=1}^n\frac{|\lb
\tilde{A}u,B_ku\rb|^2}{|\lb \tilde{A}u,u\rb|^2}\,ds},t\in [0,T]
\end{equation}
Put $\tau=0$ in the equality \eqref{eqn:NS_t-11''}. Hence,
\begin{equation}\label{eqn:NS_t-11'''}
\liminfl_{\eps\to 0}\widetilde{\widetilde{\Lambda}}_{\eps}(t)\leq
e^{\intl_{\tau}^t(n^2(r)+K_2(r)+K_6(r))dr}
e^{-2\intl_{\tau}^t\suml_{k=1}^n\frac{\lb \tilde{A}u,B_ku\rb}{|\lb
\tilde{A}u,u\rb|}\, dw^k(s)
-2\intl_{\tau}^t\suml_{k=1}^n\frac{|\lb
\tilde{A}u,B_ku\rb|^2}{|\lb
\tilde{A}u,u\rb|^2}\,ds}\frac{\widetilde{\widetilde{\Lambda}}_{\eps}(0)}{M_\eps(t)},t\in
[0,T].
\end{equation}
Notice that
$$
\widetilde{\widetilde{\Lambda}}(t)=\liminfl_{\eps\to
0}\widetilde{\widetilde{\Lambda}}_{\eps}(t), t\in [0,\tau_0).
$$
Thus
\begin{equation}\label{eqn:NS_t-11''''}
\widetilde{\widetilde{\Lambda}}(t)\leq
e^{\intl_{\tau}^t(n^2(r)+K_2(r)+K_6(r))dr}
e^{-2\intl_{\tau}^t\suml_{k=1}^n\frac{\lb \tilde{A}u,B_ku\rb}{|\lb
\tilde{A}u,u\rb|}\, dw^k(s)
-2\intl_{\tau}^t\suml_{k=1}^n\frac{|\lb
\tilde{A}u,B_ku\rb|^2}{|\lb
\tilde{A}u,u\rb|^2}\,ds}\frac{\widetilde{\widetilde{\Lambda}}(0)}{M_\eps(t)},t\in
[0,\tau_0).
\end{equation}
It follows from assumptions \textbf{(AC3)}, \textbf{(AC7)} that
RHS is uniformly bounded w.r.t. $t\in [0,T]$ a.s. for any
$\eps>0$. Hence, the result follows.

\end{proof}
%\begin{remark}
%Under the additional assumption that operators $B_k$,
%$k=1,\ldots,n$ are antisymmetric and satisfy, for some constants
%$C_1,C_2>0$,
%$$
%\suml_k|B_k w|\leq C_1\Vert w\Vert+C_2|w|, \; w\in V,
%$$
%Theorem \ref{thm:BackwardUniqueness} remain valid when condition
%\eqref{eqn:SolutionAss-1'} is replaced by a weaker condition that
%$u\in \mathcal{M}^2(0,T;D(\hat{A}))\cap
%\mathcal{M}^{2+\delta}(0,T;V)$ for some
%$\delta>0$. %\Red{Do we need the second part of the above assumption
%%or $V$-continuity is enough?} \Blue{We need it because we use it to
%%show inequality at the bottom of page 13.}
%%\Red{What about the assumption \eqref{eqn:NonlinearityAss-n} in that
%%case? This assumption and all other assumptions remain as before.
%%Can you maybe formulate a complete results in this special case?}
%\end{remark}

\section{Proof of Theorem \ref{thm:ExistenceSpectrLimit} on  the existence of a  spectral limit}
\label{sec:proof_thm_spectral}

\begin{proof}[Proof of Theorem \ref{thm:ExistenceSpectrLimit}]
Without loss of generality we can suppose that $T_0=0$ and
$\lambda=0$ in the assumption \eqref{eqn:CoercivAss}. Otherwise,
we can replace $A$ by $A+\la \id$ and $F$ by $F-\la\id$.

Let us begin the proof with an observation that  $|u(t)|>0$ for
all $t>0$. Indeed,    otherwise by the Theorem
\ref{thm:BackwardUniqueness}  we would have that $u(0)$ is
identically $0$ what would contradict one of our assumptions.
Hence the process  $$\tilde{\Lambda}(u(t))=\frac{\lb
\tilde{A}u(t),u(t) \rb}{|u(t)|^2},\;\; t\geq 0$$ is well defined.
Let us also note that $\widetilde{\Lambda}^{\eps}$
 converges pointwise and monotonously to $\widetilde{\Lambda}$.

\begin{proof}[\textbf{Step 1:} Proof of  the existence of
the limit $\liml_{t\to\infty}\tilde{\Lambda}(t)$] $\mbox{
}$\linebreak \phantom{So}Let us fix $\tau\geq 0$. By the same
argument as in the proof of the Theorem
\ref{thm:BackwardUniqueness} we get inequality \eqref{eqn:ineq-2}
i.e. $\mathbb{P}$-a.s.
\begin{eqnarray}\label{eqn:ineq-3}
M_\eps(t)\widetilde{\widetilde{\Lambda}}_{\eps}(t)&+&\intl_{\tau}^tM_\eps(s)\frac{|(\tilde{A}-\widetilde{\widetilde{\Lambda}}_{\eps})u|^2}{|u|^2+\eps}\,ds+
2\eps\intl_{\tau}^tM_\eps(s)\frac{(\widetilde{\widetilde{\Lambda}}_{\eps})^2}{|u|^2+\eps}\,ds
\nn\\
&\leq&
M_\eps(\tau)\widetilde{\widetilde{\Lambda}}_{\eps}(\tau)+\intl_{\tau}^tK_1(s)M_\eps(s)\,ds
+\frac{C}{\eps_2}\intl_{\tau}^t(n^2(s)+K_2(s)\\
&+& K_6(s))M_\eps(s)\widetilde{\widetilde{\Lambda}}_{\eps}(s)\,ds
-\intl_{\tau}^tM_\eps(s)\suml_{k=1}^n\frac{2\lb
\tilde{A}u,B_ku\rb}{|u|^2+\eps}\, dw^k(s), t\geq\tau.
\nn%,\omega\in\Omega^\prime, \mathbb{P}(\Omega^{\prime})=1.
\end{eqnarray}
By the assumptions of Theorem \ref{thm:ExistenceSpectrLimit}
$K_1(s)=0$, $s\in[0,\infty)$. Consequently, for $t\geq\tau$,
\begin{eqnarray}
M_\eps(t)\widetilde{\widetilde{\Lambda}}_{\eps}(t)+\intl_{\tau}^tM_\eps(s)\frac{|(\tilde{A}-\widetilde{\widetilde{\Lambda}}_{\eps})u|^2}{|u|^2+\eps}\,ds+
2\eps\intl_{\tau}^tM_\eps(s)\frac{(\widetilde{\widetilde{\Lambda}}_{\eps})^2}{|u|^2+\eps}\,ds
\leq M_\eps(\tau)\widetilde{\widetilde{\Lambda}}_{\eps}(\tau)\nonumber\\
+\frac{C}{\eps_2}\intl_{\tau}^t(n^2(s)+K_2(s)+K_6(s))M_\eps(s)\widetilde{\widetilde{\Lambda}}_{\eps}(s)\,ds
-\intl_{\tau}^tM_\eps(s)\suml_{k=1}^n\frac{2\lb
\tilde{A}u,B_ku\rb}{|u|^2+\eps}\, dw^k(s).\label{eqn:ineq-2'}
\end{eqnarray}
Thus, for $t\geq\tau$,
\begin{equation}
\label{eqn:X_tEstim-1}
M_\eps(t)\widetilde{\widetilde{\Lambda}}_{\eps}(t)+\intl_{\tau}^tM_\eps(s)\frac{|(\tilde{A}-\widetilde{\widetilde{\Lambda}}_{\eps})u|^2}{|u|^2+\eps}\,ds\leq
M_\eps(\tau)\widetilde{\widetilde{\Lambda}}_{\eps}(\tau)e^{\intl_{\tau}^t(n^2(s)+K_2(s)+K_6(s))\,ds}-L_{\tau}^{\eps}(t).
\end{equation}
where,
\begin{equation}
L_{\tau}^{\eps}(t)=\intl_{\tau}^te^{\intl_s^t(n^2(r)+K_2(r)+K_6(r))dr}M_\eps(s)\suml_{k=1}^n\frac{2\lb
\tilde{A}u,B_ku\rb}{|u|^2+\eps}\, dw^k(s),\;\;t\geq\tau.
\end{equation}
Let us  denote, for $t\geq 0$,
\begin{eqnarray*}
S_{\eps}(t)&=&e^{-\intl_0^t(n^2(r)+K_2(r)+K_6(r))dr}M_\eps(t)\widetilde{\widetilde{\Lambda}}_{\eps}(t),\\
N_{\eps}(t)&=&e^{-\intl_0^t(n^2(r)+K_2(r)+K_6(r))dr}\intl_{\tau}^tM_\eps(s)\frac{|(\tilde{A}-\widetilde{\Lambda}_{\eps})u|^2}{|u|^2+\eps}\,ds.
\end{eqnarray*}

Multiplying inequality \eqref{eqn:X_tEstim-1} by
$e^{-\intl_0^t(n^2(r)+K_2(r)+K_6(r))dr}$  we infer that
\begin{equation}
\label{eqn:S_t-1} S_{\eps}(t)+N_{\eps}(t)\leq
S_{\eps}(\tau)-2\intl_{\tau}^te^{-\intl_0^s(n^2(r)+K_2(r)+K_6(r))dr}M_\eps(s)\suml_{k=1}^n\frac{2\lb
\tilde{A}u,B_ku\rb}{|u|^2+\eps}\, dw^k(s).
\end{equation}
Therefore, by the definition of $S_{\eps}$, we infer that
\begin{equation}\label{eqn:S_t-1'}
S_{\eps}(t)+N_{\eps}(t)\leq
S_{\eps}(\tau)-2\intl_{\tau}^tS_{\eps}(s)\suml_{k=1}^n\frac{\lb
\tilde{A}u,B_ku\rb}{|\lb \tilde{A}u,u\rb|}\, dw^k(s).
\end{equation}
Since $N_{\eps}\geq 0$, by the comparison principle for the one
dimensional diffusions, see e.g. Theorem 1.1 p. 352 in
\cite{[Ikeda-1981]}, we have for $t\geq \tau$, $\mathbb{P}$-a.s.
\begin{equation}\label{eqn:S_t-1''}
S_{\eps}(t)\leq
S_{\eps}(\tau)e^{-2\intl_{\tau}^t\suml_{k=1}^n\frac{\lb
\tilde{A}u,B_ku\rb}{|\lb \tilde{A}u,u\rb|}\, dw^k(s)
-2\intl_{\tau}^t\suml_{k=1}^n\frac{|\lb
\tilde{A}u,B_ku\rb|^2}{|\lb \tilde{A}u,u\rb|^2}\,ds}.
\end{equation}
%Since $|u(t)|>0$, $t\geq 0$ we can put $\eps=0$ in
%\eqref{eqn:S_t-1'}.
Let us observe that inequality \eqref{eqn:S_t-1''} makes sense
when $\eps=0$. Indeed, as already mentioned before, $|u(t)|>0$ for
all $t>0$ $\mathbb{P}$-a.s.. Suppose that we can show that there
exist sequence $\{\eps_l\}_{l=1}^{\infty}$, $\eps_l\to 0$ such
that $S_{\eps_l}(s)\to S_0(s)$, $s\in [0,t]$ $\mathbb{P}$-a.s..
Then we will be able to conclude that inequality
\eqref{eqn:S_t-1''} holds with $\eps=0$.

Thus it is enough to show that
$\xi_t^{\eps}=\suml_{k=1}^n\intl_0^t\frac{<B_k
u(s),u(s)>}{|u(s)|^2+\eps}dw^k(s)$ converges to
$\xi_t=\suml_{k=1}^n\intl_0^t\frac{<B_k
u(s),u(s)>}{|u(s)|^2}dw^k(s)$ as $\eps\to 0$ in probability
uniformly on any finite interval $t\in[0,m]$, $m\in\mathbb{N}$.
Indeed, in this case there exist subsequence
$\{\eps_l\}_{l=1}^{\infty}$, $\eps_l\to 0$ such that
$\xi^{\eps_l}\to\xi$ as $l\to\infty$ uniformly on any finite
interval with probability $1$ and, therefore, $M_{\eps_l}\to M_0$
$\mathbb{P}$-a.s.. Convergence $S_{\eps_l}\to S_0$ uniformly on
any finite interval with probability $1$ as $l\to\infty$ follow.

We will show convergence of martingales $\xi_{\cdot}^{\eps}$ to
$\xi_{\cdot}$ in $L^2(\Omega\times [0,m])$ for any
$m\in\mathbb{N}$. Convergence in probability follows from Doob's
inequality. Notice that $\suml_{k=1}^n\frac{<B_k
u(t),u(t)>}{|u(t)|^2+\eps},t\in[0,\infty)$ converges a.s. to
$\suml_{k=1}^n\frac{<B_k u(t),u(t)>}{|u(t)|^2},t\in[0,\infty)$.
Indeed, $|u(t)|>0$ for all $t\geq 0$ $\mathbb{P}$-a.s..
Furthermore, by Assumption \textbf{(AC3)}
\begin{multline*}
\supl_{t}|\suml_{k=1}^n\left(\frac{<B_k
u(t),u(t)>}{|u(t)|^2+\eps}-\frac{<B_k u(t),u(t)>}{|u(t)|^2}\right)|=\\
\supl_t|\suml_{k=1}^n\frac{<B_k
u(t),u(t)>}{|u(t)|^2+\eps}||\frac{\eps}{|u(t)|_H^2+\eps}|\leq
|\phi(\cdot)|_{L^{\infty}}<\infty.
\end{multline*}
Hence, by the Lebesgue Dominated Convergence Theorem,
$\suml_{k=1}^n\frac{<B_k u(t),u(t)>}{|u(t)|^2+\eps}\to
\suml_{k=1}^n\frac{<B_k u(t),u(t)>}{|u(t)|^2}$ as $\eps\to 0$ in
$L^2(\Omega\times [0,m])$ for any $m\in\mathbb{N}$. Thus we have
shown that inequality \eqref{eqn:S_t-1''} holds with $\eps=0$,
i.e. that for $t\geq \tau$, $\mathbb{P}$-a.s.
\begin{equation}\label{eqn:S_t-2}
S(t)\leq S(\tau)e^{-2\intl_{\tau}^t\suml_{k=1}^n\frac{\lb
\tilde{A}u,B_ku\rb}{|\lb \tilde{A}u,u\rb|}\, dw^k(s)
-2\intl_{\tau}^t\suml_{k=1}^n\frac{|\lb
\tilde{A}u,B_ku\rb|^2}{|\lb \tilde{A}u,u\rb|^2}\,ds},
\end{equation}
where $S(t)=S_{0}(t)$, $N(t)=N_{0}(t)$, $t\geq 0$.

Denote
$$
\vartheta_{\tau}(t)=e^{-2\intl_{\tau}^t\suml_{k=1}^n\frac{\lb
\tilde{A}u,B_ku\rb}{|\lb \tilde{A}u,u\rb|}\, dw^k(s)
-2\intl_{\tau}^t\suml_{k=1}^n\frac{|\lb
\tilde{A}u,B_ku\rb|^2}{|\lb \tilde{A}u,u\rb|^2}\,ds},\,\, t\geq
\tau.
$$
Let us note that a process $(\vartheta_{\tau}(t))_{t\geq \tau}$ is
a local martingale. Moreover it  is a uniformly integrable
martingale. Indeed, in view of the assumption
\eqref{eqn:B_kFirstOrder}, we infer that
\begin{equation}
\supl_{t\geq \tau}\mathbb{E}[\vartheta_{\tau}(t)^{1+\delta}]\leq
\mathbb{E}e^{2(\delta^2+\delta)\intl_{\tau}^{\infty}\suml_k|C_1^k(s)|^2\,
ds}< \infty.
\end{equation}
Hence, by the Doob Martingale Convergence Theorem
\ref{thm:DoobConvergence} the following limit exists
$\mathbb{P}$-a.s. (and in $L^1(\mathbb{P})$)
$$
\vartheta_{\tau}(\infty):=\liml_{t\to\infty}\vartheta_{\tau}(t),
$$
and $\mathbb{E}\vartheta_{\tau}(\infty)=1<\infty$. Therefore,
$\vartheta_{\tau}(\infty)<\infty$ $\mathbb{P}$-a.s.. Furthermore,
$\vartheta_{\tau}(\infty)>0$ $\mathbb{P}$-a.s.. Indeed, by Fatou
Lemma
$$
\mathbb{E}\left[\vartheta_{\tau}(\infty)^{-1}\right]\leq
\liml_{t\to\infty}\mathbb{E}\left[\vartheta_{\tau}(t)^{-1}\right]\leq\mathbb{E}e^{4\intl_{\tau}^{\infty}\suml_k|C_1^k(s)|^2\,
ds} <\infty.
$$
Consequently, by Lemma \ref{lem:Cocycle_lem} we infer that
\begin{equation*}
\liml_{\tau\to\infty}\vartheta_{\tau}(\infty)=1,\mathbb{P}\mbox{-a.s.}.
\end{equation*}
Thus,
\begin{equation*}
\limsup\limits_{t\to\infty}S(t)\leq
S(\tau)\vartheta_{\tau}(\infty).
\end{equation*}
Since the above holds for any $\tau>0$, we infer that
$\mathbb{P}-\mbox{a.s.}$
\begin{equation*}
\limsup\limits_{t\to\infty}S(t)\leq\liminf\limits_{\tau\to\infty}S(\tau)\vartheta_{\tau}(\infty)=\liminf\limits_{\tau\to\infty}S(\tau).
\end{equation*}
We infer that the following limit exists $\mathbb{P}-\mbox{a.s.}$
\begin{equation}
\liml_{t\to\infty}S(t)=\liml_{t\to\infty}e^{-\intl_0^t(n^2(\tau)+K_2(\tau)+K_6(\tau))d\tau}M_{0}(t)\widetilde{\Lambda}(t)
.\label{eqn:S_tLimit-1}
\end{equation}
Moreover, since $n\in L^2(0,\infty)$, $K_2\in L^1(0,\infty)$ and
$K_6\in L^1(0,\infty)$, the following limit exists
$\mathbb{P}-\mbox{a.s.}$
\begin{equation}
\liml_{t\to\infty}e^{-\intl_0^t(n^2(\tau)+K_2(\tau)+K_6(\tau))d\tau}=e^{-\intl_0^{\infty}(n^2(\tau)+K_2(\tau)+K_6(\tau))d\tau}>0.\label{eqn:S_tLimit-2}
\end{equation}
Furthermore, from Assumption \textbf{(AC3)} it follows that
$(M_{0}(t))_{t\geq 0}$ is uniformly integrable exponential
martingale and therefore "as above" exists
$\mathbb{P}-\mbox{a.s.}$
\begin{equation}
M_{0}(\infty):= \liml_{t\to\infty}M_{0}(t)\neq
0.\label{eqn:S_tLimit-3}
\end{equation}
Combining \eqref{eqn:S_tLimit-1}, \eqref{eqn:S_tLimit-2} and
\eqref{eqn:S_tLimit-3} we infer that the following limit exists
$\mathbb{P}$-a.s.
\begin{equation}
\widetilde{\Lambda}(\infty):=\liml_{t\to\infty}\widetilde{\Lambda}(t).\label{eqn:S_tLimit-4}
\end{equation}

%Since for each $t\geq 0$,     $\widetilde{\Lambda}^{\eps}(t)\todown \widetilde{\Lambda}(t)$ as $\eps\todown 0$, the function
%$\eps\mapsto \widetilde{\Lambda}_{\eps}^{\infty}$ is weakly decreasing. Since on the other $\widetilde{\Lambda}_{\eps}^{\infty}\geq 0$, we infer that  limit of  $\widetilde{\Lambda}_{\eps}^{\infty}$ as $\eps\todown 0$ exists.
%\begin{equation}
%\liml_{\eps\todown 0}\liml_{t\to\infty}\widetilde{\Lambda}^{\eps}(t)=
%\liml_{\eps\todown 0}\widetilde{\Lambda}_{\eps}^{\infty}=\widetilde{\Lambda}(\infty),\label{eqn:S_tLimit-5}
%\end{equation}
\end{proof}

\begin{proof}[\textbf{Step 2:} Proof that $\widetilde{\Lambda}(\infty)\in \sigma(\tilde{A})$]
$\mbox{ }$\linebreak \phantom{So} We will need an estimate for
$N_{\eps}(t)$, $t\geq \tau$. Denote
$$
R_{\tau}(t)=2\intl_{\tau}^t\suml_{k=1}^n\frac{\lb
\tilde{A}u(s),B_ku(s)\rb}{|\lb \tilde{A}u(s),u(s)\rb|}\, dw^k(s),
t\geq \tau.
$$
Let us observe that the process  $R_{\tau}$ martingale in the
formula \eqref{eqn:S_t-1'}. The formula \eqref{eqn:S_t-1'} can be
rewritten as follows.
\begin{equation}
S_{\eps}(t)\leq
S_{\eps}(\tau)-N_{\eps}(t)-\intl_{\tau}^tS_{\eps}(s)dR_{\tau}(s).\label{eqn:S_t-3}
\end{equation}
Denote
$$
\Psi_{\tau}(t)=e^{R_{\tau}(t)+\frac{1}{2}\lb R_{\tau}\rb (t)},\;
t\geq \tau.
$$
By the Comparison Theorem for the one dimensional diffusions, see
e.g. Theorem 1.1 p. 352 in  \cite{[Ikeda-1981]}, we have for
$t\geq \tau$,
\begin{eqnarray}
\label{eqn:N_t-1} S_{\eps}(t) &\leq &
\frac{1}{\Psi_{\tau}(t)}\Big(S_{\eps}(\tau)-\intl_{\tau}^t\Psi_{\tau}(s)dN_{\eps}(s)\Big)
=-N_{\eps}(t)+\frac{S_{\eps}(\tau)}{\Psi_{\tau}(t)}+\frac{1}{\Psi_{\tau}(t)}\intl_{\tau}^tN_{\eps}(s)d\Psi_{\tau}(s)\nonumber\\
&=&-N_{\eps}(t)+\frac{S_{\eps}(\tau)}{\Psi_{\tau}(t)}+\frac{1}{\Psi_{\tau}(t)}\Big(\intl_{\tau}^tN_{\eps}(s)\Psi_{\tau}(s)dR_{\tau}(s)
+\intl_{\tau}^tN_{\eps}(s)\Psi_{\tau}(s)d \lb R_{\tau}\rb
(s)\Big).
\end{eqnarray}
Also by the assumption \eqref{eqn:CoercivAss} with $\la=0$ we have
that $S_{\eps}(t)\geq 0$, $t\geq \tau$. Thus, we infer from
\eqref{eqn:N_t-1} that
\begin{equation}
N_{\eps}(t)\leq
\frac{S_{\eps}(\tau)}{\Psi_{\tau}(t)}+\frac{1}{\Psi_{\tau}(t)}(\intl_{\tau}^tN_{\eps}(s)\Psi_{\tau}(s)dR_{\tau}(s)
+\intl_{\tau}^tN_{\eps}(s)\Psi_{\tau}(s)d\lb R_{\tau}\rb
(s)).\label{eqn:N_t-2}
\end{equation}
Denote $K_{\eps}(t)=N_{\eps}(t)\Psi_{\tau}(t),\; t\geq \tau$. Then
we can rewrite \eqref{eqn:N_t-2} as follows
\begin{equation}
K_{\eps}(t)\leq
S_{\eps}(\tau)+\intl_{\tau}^tK_s^{\eps}dR_{\tau}(s)+\intl_{\tau}^tK_s^{\eps}d\lb
R_{\tau}\rb (s), \;t\geq \tau.
\end{equation}
Applying the Comparison Theorem we have that
\begin{equation}
K_{\eps}(t)\leq S_{\eps}(\tau)\Psi_{\tau}(t),\; t\geq
\tau.\label{eqn:N_t-3}
\end{equation}
Consequently, we have an estimate for $N^{\eps}$:
\begin{equation}
N_{\eps}(t)\leq S_{\eps}(\tau),\; t\geq \tau.\label{eqn:N_t-4}
\end{equation}
We infer from \eqref{eqn:N_t-4} that $\mathbb{P}-\mbox{a.s.}$
\begin{equation}
\liml_{t\to\infty}N_{\eps}(t)\leq S_{\eps}(\tau).
\end{equation}
i.e. $\mathbb{P}-\mbox{a.s.}$
\begin{equation}
e^{-\intl_0^{\infty}(n^2(\tau)+K_2(\tau)+K_6(\tau))d\tau}\intl_{\tau}^{\infty}M_\eps(s)\frac{|(\tilde{A}-\widetilde{\Lambda}_{\eps})u(t)|^2}{|u|^2+\eps}\,ds\leq
e^{-\intl_0^{\tau}(n^2(\tau)+K_2(\tau)+K_6(\tau))d\tau}\widetilde{\widetilde{\Lambda}}_{\eps}(\tau).\label{eqn:Aux-4}
\end{equation}
%Notice  that the right hand side of inequality \eqref{eqn:Aux-4} is
%uniformly bounded w.r.t. $\eps$ and $|u(t)|>0$, $t>0$.
As we have already mentioned $|u(t)|>0$, $t>0$. Hence the right
hand side of inequality \eqref{eqn:Aux-4} is uniformly bounded
w.r.t. $\eps$. Therefore, by Fatou Lemma $\mathbb{P}-\mbox{a.s.}$
\begin{equation}
\intl_{\tau}^{\infty}M(s)\frac{|(\tilde{A}-\widetilde{\Lambda}(s))u(s)|^2}{|u(s)|^2}\,ds<
\infty.\label{eqn:SpectrIntCond}
\end{equation}
Let $\psi:\mathbb{R}\ni t\mapsto\frac{u(t)}{|u(t)|}\in H$. It
follows from \eqref{eqn:SpectrIntCond} that there exists sequence
$\{t_j\}_{j=1}^{\infty}:t_j\to\infty$ as $j\to\infty$ and
$(\tilde{A}-\widetilde{\Lambda}(t_j))\psi(t_j)\to 0$ in $H$.
Therefore, by \eqref{eqn:S_tLimit-4}
$h_j=(\tilde{A}-\widetilde{\Lambda}(\infty))\psi(t_j)\to 0$ in
$V^\prime$ as $j\to\infty$. If $\widetilde{\Lambda}(\infty)\notin
\sigma(\tilde{A})$  then
$(\tilde{A}-\widetilde{\Lambda}(\infty))^{-1}\in\mathcal{L}(V^\prime,V)$.
Since $h_j\to 0$ in $V^\prime$ we have
$\psi(t_j)=(\tilde{A}-\widetilde{\Lambda}(\infty))^{-1}h_j\to 0$
in $H$. This is contradiction with the fact that $|\psi(t)|=1$.
This completes the argument in \textbf{Step 2}.
\end{proof}
The proof of Theorem \ref{thm:ExistenceSpectrLimit} is now
complete.
\end{proof}

\section{Applications and Examples}
\label{sec:appl}

Now we will show how to apply Theorems
\ref{thm:BackwardUniqueness} and \ref{thm:ExistenceSpectrLimit} to
certain linear and nonlinear SPDEs.
\subsection{Backward Uniqueness for Linear SPDEs.}
 We will consider following equation: %\Red{If this is a linear equation why doo we have $F(s,u)$ below?}
%\Green{We didn't assume anywhere that $F(s,u)$ is necessarily
%nonlinearity. Condition on $F(s,u)$ is much more relaxed comparing to
%condition on $A$. So, as you can notice in the example
%\ref{exm:BackUniq-1} below, I put in $F(s,u)$ lower order terms.}
\begin{eqnarray}
\,\,\,\,\,\,\,\,\,\,du+(Au+F(t)u)\,dt+\suml_{k=1}^nB_ku\,dw^k_t=f\,dt+\suml_{k=1}^ng_k\,
dw^k_t,\label{eqn:BasicAbstractEq}
\end{eqnarray}
where $f\in \mathcal{M}^2(0,T;V)$, $g\in
\mathcal{M}^2(0,T;D(\hat{A}))$ and the
 operators $A$, $B_k$, $k=1,\ldots,n$ satisfy the same
assumptions of Theorem \ref{thm:BackwardUniqueness}. We will
suppose that $F\in L^2(0,T;\mathcal{L}(V,H))$. Then, we notice
that assumption \eqref{eqn:NonlinearityAss} is satisfied with
$n=|F(\cdot)|_{\mathcal{L}(V,H)}$. Applying Theorem
\ref{thm:BackwardUniqueness} we have the following result:
\begin{theorem}\label{thm:BackwardUniqueness-2}
Let $u_1$, $u_2$ be two solutions of \eqref{eqn:BasicAbstractEq},
such that for some $\delta_0>0$
\begin{eqnarray}
u_1,u_2\in %L^2(\Omega,C(0,T,V))\cap%
\mathcal{M}^2(0,T;D(\hat{A}))\cap
L^{2+\delta_0}(\Omega,C([0,T];V)).\label{eqn:IntegrCond}
\end{eqnarray}
Then if $u_1(T)=u_2(T)$, $\mathbb{P}$-a.s., $u_1(t)=u_2(t),\;
t\in[0,T]$ $\mathbb{P}$-a.s..
\end{theorem}
\begin{proof}
Denote $u=u_1-u_2$. Applying Theorem \ref{thm:BackwardUniqueness}
to the process $u$ we immediately get the result.
\end{proof}
\begin{example}\label{exm:BackUniq-0}
Assume $a^{i,j}\in C_{t,x}^{1,1}([0,T]\times \Rn)$,
$i,j=1,\ldots,n$ and $b_k,c,c_l:[0,T]\times \Rn\to \mathbb{R}$,
$k=1,\ldots,n$, $l=1,\ldots,m$ are measurable functions. Denote
$\tilde{A}_0(\cdot)=(a^{i,j}(\cdot))_{i,j=1}^n$. Assume that the
following inequalities are satisfied
$$
C_1\id\leq \tilde{A}_0\leq C_2\id, \mbox{ for some }0<C_1\leq
C_2<\infty,
$$
$$
\intl_0^T|\partial_t\tilde{A}_0|ds\leq C_3\id, \mbox{ for some
}C_3>0
$$
$$
\supl_{x\in \Rn}\suml_{k}|b_k(t,\cdot)|+|c(t,\cdot)|\in L^2(0,T).
$$
and
$$
c_l\in L^{\infty}([0,T]\times\Rn),\nabla c_l\in
L^1(0,T;L^{\infty}(\Rn)),\triangle c_l\in
L^2(0,T;L^{\infty}(\Rn)).
$$
Then the equation
\begin{equation*}
  du = \left(\suml_{i,j}\frac{\partial}{\partial x_i}\left(a^{ij} \frac{\partial u}{\partial x_j}\right)+\suml_kb_k(t,\cdot)\frac{\partial u}{\partial x_k}+c(t,\cdot)u+f\right)\,dt+\suml_{l=1}^m\left(c_l(t,\cdot)u+g_l\right)\circ \, dw^l_t,
\end{equation*}
where stochastic integral is in Stratonovich sense, satisfies
conditions of the Theorem \ref{thm:BackwardUniqueness-2}. Indeed,
we have in this case that
$$
\tilde{A}=-\suml_{i,j}\frac{\partial}{\partial x_i}\left(a^{ij}
\frac{\partial}{\partial x_j}\right) ,
F(t)=-\suml_kb_k(t,\cdot)\frac{\partial}{\partial
x_k}-c(t,\cdot),B_l=c_l,
$$
$$
H=L^2(\Rn), V=H_0^{1,2}(\Rn).
$$
\end{example}
\begin{example}\label{exm:BackUniq-1}
Assume $b_k,c,\sigma_k\in C_{t,x}^{0,1}([0,T]\times \Rn)$,
$k=1,\ldots,n$ and following inequalities are satisfied:
$$
\supl_{x\in \Rn}\suml_{k}|\nabla\sigma_k(t,\cdot)|^2\in
L^{\infty}(0,T),
$$
$$
\supl_{x\in \Rn}\suml_{k}|b_k(t,\cdot)|+|c(t,\cdot)|\in L^2(0,T).
$$
Then the equation
\begin{equation*}
  du = \left(\Delta  u+\suml_kb_k(t,\cdot)\frac{\partial u}{\partial x_k}+c(t,\cdot)u+f\right)\,dt+\suml_k\left(\sigma_k(t,\cdot)\frac{\partial u}{\partial x_k}+g_k\right)\circ \, dw^k_t,
\end{equation*}
where stochastic integral is in Stratonovich sense, satisfies
conditions of the Theorem \ref{thm:BackwardUniqueness-2}. Indeed,
we have in this case that
$$
\tilde{A}=-\Delta ,
F(t)=-\suml_kb_k(t,\cdot)\frac{\partial}{\partial
x_k}-c(t,\cdot),B_k=\sigma_k(t,\cdot)\frac{\partial}{\partial
x_k},
$$
$$
H=L^2(\Rn), V=H_0^{1,2}(\Rn).
$$
We need to check only assumption \eqref{eqn:CWeakBoundAss}. Other
conditions are trivial. We have
\begin{eqnarray*}
\suml_k([\tilde{A},B_k]u,B_ku)&=&\suml_k\intl_{\Rn}\sigma_k\frac{\partial
u}{\partial x_k}\big(\sigma_k\frac{\partial\Delta  u}{\partial
x_k}-\Delta (\sigma_k\frac{\partial u}{\partial x_k})\big)\,dx
\\
&=&\suml_k\intl_{\Rn}|\nabla \sigma_k|^2|\frac{\partial
u}{\partial x_k}|^2dx\leq \supl_{x\in
\Rn}\suml_{k}|\nabla\sigma_k(t,\cdot)|^2 \Vert u\Vert^2.
\end{eqnarray*}
The existence of a regular solution has been established in
\cite{[Pardoux_1979]}.
\end{example}
\begin{remark}
Instead of the Laplacian one can consider a second order time
dependent operator
$A(t)=\suml_{ij}a^{ij}(t)\frac{\partial^2}{\partial x_i\partial
x_j}$ where matrix $a=(a^{ij}):[0,T]\to \Rn$ is uniformly (w.r.t.
$t$) positively definite.
\end{remark}
\begin{example}
Let $H$ be the real separable Hilbert space, $A:D(A)\subset H\to
H$ be a strictly negative linear operator on $H$;
$V=D((-A)^{1/2})$ be the Hilbert space endowed with the natural
norm. Identifying $H$ with its dual we can write $V\subset
H\subset V^\prime$. Let also $B:V\times V\to V^\prime$ be a
bilinear continuous operator and $b_k\in \mathbb{R}$,
$k=1,\ldots,n$ are given. Assume that $B$ satisfies
\begin{equation}
|B(u,v)|\leq C|u|^{1/2}|Au|^{1/2}\Vert v\Vert,u\in D(A),v\in V
\end{equation}
and
\begin{equation}
|B(u,v)|\leq C|u||v|_{D(A)},u\in H,v\in D(A).
\end{equation}
Then equation
\begin{equation*}
du = (A u +
B(u,u_{stat}(t))+B(u_{stat}(t),u))dt+\suml_{k=1}^nb_ku\circ dw^k_t
\end{equation*}
where $u_{stat}\in L^2(0,T;D(A))$ $\mathbb{P}$-a.s. satisfies
conditions of the Theorem \ref{thm:BackwardUniqueness}. Indeed, it
is enough to put
$$
\tilde{A}=A,F(t)=B(u_{stat}(t),\cdot)+B(\cdot,u_{stat}(t)),B_k=b_k\cdot
$$
In particular, in this scheme falls linearisation around solution
$u_{stat}$ of two dimensional stochastic Navier-Stokes equation
with multiplicative noise (see \cite{Cr-Fl-94}).
\end{example}
%\begin{example}
%
%\end{example}
\subsection{Backward Uniqueness for SPDEs with a quadratic nonlinearity}
Assume that the linear operators $A$, $B_k$, $k=1,\ldots,n$
satisfy the same assumptions as in the Theorem
\ref{thm:NonlinearBackwardUniqueness}, $B\in\mathcal{L}(V\times
V,V^\prime)$, $R\in\mathcal{L}(V,H)$ and
\begin{equation}
|B(u,v)|+|B(v,u)|\leq K\Vert u\Vert|Av|,\, u\in V,v\in
D(\hat{A}).\label{eqn:QuadrNonlinCond}
\end{equation}
Assume that $f\in L^2(0,T;H)$ and $g_k\in \mathcal{M}^2(0,T;H)$,
$k=1,\ldots,n$. Consider the following problem.
\begin{eqnarray}
du&+&(Au+B(u,u)+R(u))\,dt+\suml_{k=1}^nB_ku\,dw^k_t=f\,dt+\suml_{k=1}^ng_k\,
dw^k_t,\label{eqn:BasicAbstractEq-2}
\end{eqnarray}
Applying Theorem \ref{thm:NonlinearBackwardUniqueness} we have the
following result.
\begin{theorem}\label{thm:BackwardUniqueness-3}
Suppose that  $u_1$, $u_2$ are  two solutions of
\eqref{eqn:BasicAbstractEq-2}, such that for some  $\delta_0>0$
and $i=1,2$,
\begin{eqnarray}
&&u_i \in %L^2(\Omega,C(0,T,V))\cap%
\mathcal{M}^2(0,T;D(\hat{A}))\cap
L^{2+\delta_0}(\Omega,C([0,T];V)),\label{eqn:ChBIntegrCond-2}
\end{eqnarray}
If $u_1(T)=u_2(T)$, $\mathbb{P}$-a.s., then for all $t\in[0,T]$,
$u_1(t)=u_2(t)$, $\mathbb{P}$-a.s..
\end{theorem}

\begin{proof}[Proof of Theorem \ref{thm:BackwardUniqueness-3}.]
We denote $u=u_1-u_2$. Then $u\in
\mathcal{M}^2(0,T;D(\hat{A}))\cap
L^{2+\delta_0}(\Omega,C([0,T];V))$ and $u$ is a solution to
\begin{equation}
du+(Au+B(u_1,u)+B(u,u_2)+R(u))\,dt+\suml_{k=1}^nB_ku\,dw^k_t=0.
\end{equation}
By  the assumption \eqref{eqn:QuadrNonlinCond} it follows that
\begin{eqnarray*}
|B(u_1,u)+B(u,u_2)+R(u)| &\leq &
[\Vert R\Vert_{\mathcal{L}(V,H)}+K(|Au_1|+|Au_2|)]\Vert u\Vert\\
&&\hspace{-2truecm}\lefteqn{ \leq C[\Vert
R\Vert_{\mathcal{L}(V,H)}+K(|\tilde{A}u_1|+|\tilde{A}u_2|)]\Vert
u\Vert.}
\end{eqnarray*}
Therefore,  by \eqref{eqn:ChBIntegrCond-2} we have that $n=\Vert
R\Vert_{\mathcal{L}(V,H)}+K(|\tilde{A}u_1|+|\tilde{A}u_2|)$
satisfy assumption
(\ref{eqn:NNonlinearityAss}-\ref{eqn:NNonlinearityAss-n}) and
Theorem \ref{thm:NonlinearBackwardUniqueness} applies to $u$.
\end{proof}
The stochastic Navier-Stokes equations with multiplicative noise
fit
into the framework described above.%\Red{And what? Can we formulate a
\begin{example}
Let $d=2$ or $d=3$,
$H=\{f\in\mathbb{L}^2(\mathbb{T}^d,\mathbb{R}^d):\diver f=0\}$,
$V=W^{1,2}(\mathbb{T}^d,\mathbb{R}^d)\cap H$, $\nu>0$, $A$ is a
Stokes operator, $P:\mathbb{L}^2(\mathbb{T}^d,\mathbb{R}^d)\to H$
is a projection on divergence free vector fields, $B_k\in
L^{\infty}([0,T]\times\mathbb{T}^d,\mathcal{L}(\mathbb{R}^d,\mathbb{R}^d))$,
$k=1,\ldots,n$, $\{W_t^k\}_{k=1}^n$ is a sequence of independent
one dimensional Wiener processes on
$(\Omega,\mathcal{F},\{\mathcal{F}_t\}_{t\geq 0},\mathbb{P})$.
Assume also that
$$
\suml_{k=1}^n|B_k|_{L^{\infty}([0,T]\times\mathbb{T}^d,\mathcal{L}(\mathbb{R}^d,\mathbb{R}^d))}+
|\nabla
B_k|_{L^{\infty}([0,T]\times\mathbb{T}^d,\mathcal{L}(\mathbb{R}^d,\mathbb{R}^d))}+
|\triangle
B_k|_{L^{\infty}([0,T]\times\mathbb{T}^d,\mathcal{L}(\mathbb{R}^d,\mathbb{R}^d))}<\infty.
$$
Let also $f\in L^2(0,T;H)$ and $g_k\in \mathcal{M}^2(0,T;H)$,
$k=1,\ldots,n$. Assume that $u_1,u_2\in\mathcal{M}^2(0,T;D(A))\cap
L^{2+\delta_0}(\Omega,C([0,T];V))$ are two solutions of equation:
\begin{equation}
du(t)+P((u(t)\nabla)u(t))dt= (\nu
Au(t)+f(t))dt+\suml_{k=1}^n(PB_k(t)u(t)+g_k(t))\circ dW_t^k,t\in
[0,T],\label{eqn:2D_NSE_Mult_Noise}
\end{equation}
where stochastic integral is in a Stratonovich sense. Then the
assumptions of Theorem \ref{thm:BackwardUniqueness-3} are
satisfied and backward uniqueness holds i.e. if $u_1(T)=u_2(T)$,
$\mathbb{P}$-a.s., then for all $t\in[0,T]$, $u_1(t)=u_2(t)$,
$\mathbb{P}$-a.s..
\end{example}

\subsection{Existence of the spectral limit.}
\begin{example} As it is usual, we denote by $\mathbb{T}^n$ the $n$-dimensional torus. We put
Let $H=L^2(\mathbb{T}^n,\mathbb{R}^n)$ and
$V=W^{1,2}(\mathbb{T}^n,\mathbb{R}^n)$. Assume that $u\in
\mathcal{M}^2(0,T;W^{2,2}(\mathbb{T}^n,\mathbb{R}^n))\cap
L^{2+\delta_0}(\Omega,C([0,T];V))$ is a unique solution of
equation:
\begin{eqnarray*}
  du^k &=& (\suml_{i,j=1}^n\frac{\partial}{\partial x_i}(a^{ij}(\cdot)\frac{\partial u^k}{\partial x_j})+\suml_{l=1}^nb_l(t,\cdot)\frac{\partial u^k}{\partial x_l}+\suml_{l=1}^nc^{kl}(t,\cdot)u^l)dt
  +\suml_{m,l=1}^nh_k^{ml}(t)u^l\circ dw_t^m \\
  u(0) &=& u_0\in V,k=1,\ldots,n
\end{eqnarray*}
Here we assume that matrix $\tilde{A}=(a^{ij}(x))_{i,j=1}^n,
x\in\mathbb{T}^n$ is strictly positive definite, $a^{ij}\in
L^{\infty}(\mathbb{T}^n)$, $i,j=1,\ldots,n$; $h_k^{ml}\in
L^{\infty}([0,\infty)\times \mathbb{T}^n)$, $k,m,l=1,\ldots,n$,
$$
\intl_{0}^{\infty}(|h(s)|_{L^{\infty}(\mathbb{T}^n)}^2+|\nabla
h(s)|_{L^{\infty}(\mathbb{T}^n)}^2+|\triangle
h(s)|_{L^{\infty}(\mathbb{T}^n)}^2)ds<\infty,
$$
$$
\intl_{0}^{\infty}(|b(s)|_{L^{\infty}(\mathbb{T}^n)}+|c(s)|_{L^{\infty}(\mathbb{T}^n)})ds<\infty.
$$
Then assumptions of the Theorem \ref{thm:ExistenceSpectrLimit} are
satisfied and the spectral limit exists.
\end{example}

\appendix

%\section{Some useful lemmas}
%\begin{lemma}
%Let $R()$
%
%\end{lemma}
\section{Some useful known results}
\label{sec:app} We present here, for convenience of readers, some
standard definitions, lemmas and theorems used in the article. We
follow here book \cite{[Oksendal-2003]}, appendix C and references
therein.
\begin{definition}
A family $(f_j)_{j\in J}$ of measurable functions $f_j:\Omega\to
\Rnu$ is called uniformly integrable if
$$
\liml_{M\to \infty}\left(\supl_{j\in
J}\intl_{|f_j|>M}|f_j|\,d\mathbb{P}\right)=0
$$
\end{definition}
\begin{definition}
An increasing and convex function $\psi:[0,\infty)\to[0,\infty)$
is called a uniform integrability test function if and only if
$\psi$ is  and
$$
\liml_{x\to\infty}\frac{\psi(x)}{x}=\infty.
$$
\end{definition}
\begin{example}
$\psi(x)=x^p,p>1$, $x\geq 0$.
\end{example}
\begin{theorem}\label{thm:UniformIntegrabilityCrit}
The family $(f_j)_{j\in J}$ of measurable functions $f_j:\Omega\to
\Rnu$ is uniformly integrable if and only if there is a uniform
integrability test function $\psi$ such that
$$
\supl_{j\in J}\intl\psi(|f_j|)\,d\mathbb{P} <\infty.
$$
\end{theorem}
\begin{theorem}\label{thm:LimitBehaviourUniformInt}
Suppose $\{f_k\}_{k\geq 1}$ is a family of measurable functions
$f_k:\Omega\to\Rnu$ such that
$$
\liml_{k\to\infty}f_k(\omega)=f(\omega),\mbox{ for a.a. }\omega.
$$
Then the following are equivalent:
\begin{trivlist}
\item[(1)] The family $\{f_k\}_{k\geq 1}$ is uniformly integrable.
\item[(2)] $f\in L^1(P)$ and $f_k$ converges to $f$ in $L^1(P)$.
\end{trivlist}
\end{theorem}
Now we will give two applications of the notion of uniform
integrability:
\begin{proposition}[Ex. 7.12, a), p.132 of
\cite{[Oksendal-2003]}]\label{prop:MartingaleCrit} Suppose that
$\{Z_t\}_{t\in[0,\infty)}$ is a local martingale such that for
some $T>0$  the family
$$
\left\{ Z(\tau): \tau\leq T, \tau \mbox{ is a stopping time}
\right\}
$$
is uniformly integrable. Then $\{Z_t\}_{t\in[0,T]}$ is a
martingale.
\end{proposition}
\begin{theorem}[Doob's martingale convergence Theorem]\label{thm:DoobConvergence}
Let $\{Z_t\}_{t\geq 0}$ be a right continuous supermartingale.
Then the following are equivalent:
\begin{trivlist}
\item[(1)] $\{Z_t\}_{t\geq 0}$ is uniformly integrable. \item[(2)]
There exist $Z\in L^1(P)$ such that $Z_t\to Z$ $P$-a.e. and
$Z_t\to Z$ in $L^1(P)$.
\end{trivlist}
\end{theorem}

\begin{proposition}
If $t\in [0,T]$ and $H\in L^1(\Omega,\mathcal{F}_t,\mathbb{P})$,
then
\begin{equation}\label{eqn:Part1AuxInequality-2}
\mathbb{E}_{\mathbb{Q}^\eps}H=\mathbb{E}M_\eps(T)H=\mathbb{E}\left[\mathbb{E}(M_\eps(T)H|\mathcal{F}_t)\right]=
\mathbb{E}\left[H\mathbb{E}(M_\eps(T)|\mathcal{F}_t)\right]=\mathbb{E}M_\eps(t)H.
\end{equation}
\end{proposition}
\begin{lemma}\label{lem:Cocycle_lem}
Assume that function $\theta:\mathbb{R}^+\times\mathbb{R}^+\to
\mathbb{R}^+$ satisfies properties
\begin{trivlist}
\item[(i)] $\theta(s,t)=\theta(s,u)\theta(u,t)$, for any
$s,u,t\in\mathbb{R}^+$. \item[(ii)] $\theta(s,s)=1$ for any
$s\in\mathbb{R}^+$. \item[(iii)] There exist finite limit
$\theta(s)=\liml_{t\to\infty}\theta(s,t)$, $s\in\mathbb{R}^+$.
\item[(iv)] There exist $u_0\in\mathbb{R}^+$ such that
$\theta(u_0)>0$.
\end{trivlist}
Then $\liml_{s\to\infty}\theta(s)=1$.
\end{lemma}
\begin{proof}
By property (i) applied with $s=\tilde{s},u=\tilde{t},t=\tilde{s}$
and property (ii) we infer that
\begin{equation}
\theta(\tilde{s},\tilde{t})\theta(\tilde{t},\tilde{s})=\theta(\tilde{s},\tilde{s})=1,
\tilde{s},\tilde{t}\in\mathbb{R}^+.\label{eqn:Cocycle_lem_1}
\end{equation}
Therefore $\theta(s,t)>0$ for all $s,t\in\mathbb{R}^+$ and
\begin{equation}
\theta(s,t)=\theta(t,s)^{-1},t,s\in\mathbb{R}^+.\label{eqn:Cocycle_lem_2}
\end{equation}
By properties (i) and (iii) we deduce that
\begin{equation}
\theta(s)=\liml_{t\to\infty}\theta(s,u)\theta(u,t)=\theta(s,u)\theta(u),
s,u\in\mathbb{R}^+.\label{eqn:Cocycle_lem_3}
\end{equation}
Therefore, by equality \eqref{eqn:Cocycle_lem_3} with $u=u_0$ and
property (iv) we infer that
\begin{equation}
\theta(s)=\theta(s,u_0)\theta(u_0)>0,s\in\mathbb{R}^+.\label{eqn:Cocycle_lem_4}
\end{equation}
Combining equalities \eqref{eqn:Cocycle_lem_2} and
\eqref{eqn:Cocycle_lem_3} we get
\begin{equation}
\theta(s)=\frac{\theta(u)}{\theta(u,s)},u,s\in\mathbb{R}^+.
\end{equation}
Hence, by identity \eqref{eqn:Cocycle_lem_4} we infer that
$$\liml_{s\to\infty}\theta(s)=\liml_{s\to\infty}\frac{\theta(u)}{\theta(u,s)}=\frac{\theta(u)}{\theta(u)}=1.$$
\end{proof}
\begin{lemma}\label{lem:AuxRes_1}
Let $H$ be Hilbert space, $C\in \mathcal{L}(H,H)$, $\eps>0$,
$F:H\to\mathbb{R}$ is defined by
$$
F(x)=\frac{<Cx,x>}{|x|_H^2+\eps},x\in H.
$$
Then $F$ is of $C^2$-class, the $1^{\rm st}$ and $2^{\rm nd}$
derivatives of $F$ are continuous bounded functions and they are
given by the following formulas:
\begin{eqnarray}
F'(x)h_1 &=& \frac{2\lb C x,h_1\rb}{|x|^2+\eps}-\frac{2\lb C x,x\rb\lb x,h_1\rb}{(|x|^2+\eps)^2},\label{eqn:FirstDer}\\
F''(x)(h_1,h_2) &=& 2\frac{\lb C h_1,h_2\rb}{|x|^2+\eps}-4\frac{\lb C x,h_1\rb\lb x,h_2\rb}{(|x|^2+\eps)^2}\label{eqn:SecondDer}\\
&-&4\frac{\lb  C x,h_2\rb\lb x,h_1\rb}{(|x|^2+\eps)^2}-2\frac{\lb C x,x\rb\lb h_2,h_1\rb}{(|x|^2+\eps)^2}\nonumber\\
&+&8\frac{\lb C x,x\rb\lb x,h_1\rb\lb
x,h_2\rb}{(|x|^2+\eps)^3}.\nonumber
\end{eqnarray}
\end{lemma}
\begin{proof}[Proof of Lemma \ref{lem:AuxRes_1}]
Proof is omitted.
\end{proof}

\end{document}